\documentclass[12pt,reqno]{amsart}

\usepackage{amsthm,amsfonts,amssymb,euscript,mathrsfs,graphics, color, amsmath, amssymb, latexsym, marginnote}
\usepackage{mathrsfs}

\usepackage{hyperref}

\newtheorem{thm}{Theorem}[section]

\newtheorem{lem}[thm]{Lemma}

\newtheorem{prop}[thm]{Proposition}
\newtheorem{proposition}[thm]{Proposition}

\newtheorem{df}[thm]{Definition}
\newtheorem{rem}[thm]{Remark}
\newtheorem{ex}{Example}
\newtheorem{assumption}[thm]{Assumption}
 
\numberwithin{equation}{section}

\def\be#1 {\begin{equation} \label{#1}}
\newcommand{\ee}{\end{equation}}

\def\R{{\mathbb R}}

\def\e{{\varepsilon}}

\def\s{{\sigma}}

\def\l{{\lambda}}
\def\k{{\kappa}}

\def\max{\mathrm{max}}

\def\min{\mathrm{min}}

\def\supp{\,\mbox{supp}\,}

\def\jxi{\langle \xi \rangle}
\def\jeta{\langle \eta \rangle}
\def\jsig{\langle \sigma \rangle}

\def\jt{\langle t \rangle}
\def\js{\langle s \rangle}

\def\jx{\langle x \rangle}

\def\what{\widehat}
\def\wt{\widetilde}
\def\bar{\overline}
\def\eps{\epsilon}

\def\pv{\mathrm{p.v.}}

\def\wtF{\widetilde{\mathcal{F}}}
\def\whatF{\widehat{\mathcal{F}}}

\DeclareMathAlphabet{\mathpzc}{OT1}{pzc}{m}{it}

\numberwithin{equation}{section}

\makeatletter
\newcommand{\pushright}[1]{\ifmeasuring@#1\else\omit\hfill$\displaystyle#1$\fi\ignorespaces}
\newcommand{\pushleft}[1]{\ifmeasuring@#1\else\omit$\displaystyle#1$\hfill\fi\ignorespaces}
\makeatother

\definecolor{green}{rgb}{0.0, 0.5, 0.5}

\newcommand{\zz}[1]{\color{blue}  { #1} \color{black}  }

\begin{document}

\author{Pierre Germain}
\address{Courant Institute of Mathematical Sciences, 251 Mercer Street, New York 10012-1185 NY, USA}
\email{pgermain@cims.nyu.edu}

\author{Fabio Pusateri}
\address{Fabio Pusateri,  Department of Mathematics, University of Toronto, 40 St. George street, Toronto, 
  M5S 2E4, Ontario, Canada}
\email{fabiop@math.toronto.edu}

\author{Katherine Zhiyuan Zhang}
\address{Courant Institute of Mathematical Sciences, 251 Mercer Street, New York 10012-1185 NY, USA}
\email{zz3463@nyu.edu}

\title{On $1$d quadratic Klein-Gordon equations \\ with a potential 
and symmetries} 




\begin{abstract} 
This paper is a continuation of the previous work \cite{KGV1d} by the first two authors.
We focus on $1$ dimensional quadratic Klein-Gordon equations with a potential, 
under some assumptions that are less general than \cite{KGV1d}, but
allow us to present some simplifications 
in the proof of 
global existence with decay 
for small solutions.
In particular, we can propagate a stronger control on a basic $L^2$-weighted type norm
while providing some shorter and less technical proofs for some of the arguments. 
\end{abstract}

\maketitle

\markboth{P. GERMAIN, F. PUSATERI, Z. ZHANG}{ON $1$D QUADRATIC KLEIN-GORDON WITH A POTENTIAL}

\tableofcontents

\section{Introduction}

\subsection{Assumptions and main theorem}
We consider the equation
\begin{equation}
\label{maineq}
(\partial_t^2 - \partial_x^2 + V  + 1) u = a(x) u^2 + b(x) u^3
\end{equation}
under the following assumptions: the potential $V$ is even
(this assumption can be dropped in the case of generic $V$)
and rapidly decaying together with its derivatives\footnote{Smoothness and decay assumptions
are stated like this for convenience, but a finite amount of regularity and algebraic decay is sufficient.}:
$$
|\partial_x^\alpha V(x) | \lesssim_{\alpha,N} \langle x \rangle^{-N}, \qquad \mbox{for all $\alpha, N$},
$$
and the functions $a$ and $b$ have limits at $\pm \infty$, which they reach rapidly:
\begin{equation}\label{ab}
|\partial_x^\alpha (a(x) - a_{\pm \infty}) | \lesssim_{\alpha,N} \langle x \rangle^{-N}, 
\qquad |\partial_x^\alpha (b(x) - b_{\pm \infty}) | \lesssim_{\alpha,N} \langle x \rangle^{-N}, \qquad \mbox{for $\pm x >0$}.
\end{equation}
We assume that $b$ is even, and will specify additional parity assumptions on $a$ below.\footnote{These 
parity assumptions can be omitted for generic potentials.}



Furthermore, the Schr\"odinger operator
$$
H = - \partial_x^2 + V
$$
is assumed to have no eigenvalues, and to satisfy the assumption below.

\begin{assumption}  \label{A:op}
One of the three conditions below is satisfied:
\begin{itemize}
\item The potential $V$ is generic.

\item The potential $V$ is exceptional, with an even zero energy resonance, the data and $a(x)$ are odd.

\item The potential $V$ is exceptional, with an odd zero energy resonance, the data and $a(x)$ are even.
\end{itemize}
\end{assumption}
The readers can refer to Definition \ref{defV} and \eqref{f+-}--\eqref{m+-} for the definition of generic and exceptional (or non-generic) potentials.

Finally, we provide the equation with data at time zero:
$$u(t=0) = u_0, \quad \partial_t u(t=0) = u_1.$$

Our main result is then the following:

\begin{thm} \label{mainthm}
Under the above assumptions, there exists $\varepsilon_0>0$ such that, if $(u_0,u_1)$ satisfy
$$
\| (\sqrt{H+1} u_0,u_1) \|_{H^4} + \| \langle x \rangle (\sqrt{H+1} u_0,u_1) \|_{H^1} = \varepsilon < \varepsilon_0,
$$
then there exists a unique global solution of \eqref{maineq}
which decays pointwise and is globally bounded in $L^2$ type spaces: for any time $t$,
\begin{align*}
& \| (\sqrt{H+1} u(t), \partial_t u(t)) \|_{L^\infty} \lesssim \varepsilon \langle t \rangle^{-1/2} 
\\
& \| u(t) \|_{H^5} + \| \partial_t u (t) \|_{H^4} \lesssim \varepsilon \langle t \rangle^{p_0},
\end{align*}
where $p_0$ is a small\footnote{
$p_0$ can be chosen of the form $C\varepsilon^2$ for some absolute constant $C>0$,
this constraint arising from the Sobolev energy estimate.}
number.
\end{thm}

In particular, our result gives a simpler proof of stability for families of kinks of the double Sine-Gordon equation under odd
perturbations, see the discussion in Section 1.4.3 in \cite{KGV1d}.
The proof also gives a more precise description of the asymptotic behavior of $u$, 
which undergoes a type of modified scattering. To be more specific, the distorted Fourier transform 
(defined in Section \ref{STIDO}) of a suitably renormalized profile $f$ 
(defined in Section \ref{MND}, see \eqref{profile1}) satisfies the following: 
there exists an asymptotic profile $W^{\infty} = (W^{\infty}_+, W^{\infty}_-) 
\in \big(\jxi^{-3/2}L^\infty_\xi\big)^2$ such that,
for $\xi>0$,
\begin{align}\label{asy}
\begin{split}
&  \big(\wt{f}(t,\xi),\wt{f}(t,-\xi)\big)
	\\ & = S^{-1}(\xi) \exp\Big( -\frac{5i}{12}  \mathrm{diag}\big( \ell_{+\infty}^2\big| W^{\infty}_+(\xi) \big|^2, 
	\ell_{-\infty}^2\big| W^{\infty}_-(\xi) \big|^2 \big) \log t  \Big) W^{\infty}(\xi)
	+ O\big(\e_0^2 \jt^{-\delta_0}\big)
\end{split}
\end{align}
as $t \rightarrow \infty$, for some $\delta_0>0$,
where $S(\xi)$ is the scattering matrix (see Section \ref{STIDO} for its definition) 
and $\ell_{\pm \infty}$ can be determined from $a_{\pm\infty}$ and $b_{\pm\infty}$
(we have $ \ell_{\pm \infty} = a_{\pm \infty}$ when $b_{\pm\infty} = 0$).


\smallskip
\subsection{Relation to \cite{KGV1d}}
Denoting $\widetilde{f}$ for the distorted Fourier transform associated to $H$ of a function $f$, 
the assumptions on $H$ made in the previous subsection ensure that
$\widetilde{u}(0) = 0$ for $u$ solution of \eqref{maineq}; 
this was the key hypothesis in the main theorem of \cite{KGV1d} 
on global solutions and asymptotics for \eqref{maineq}. 
However, the assumptions made above Theorem \ref{mainthm} imply more: 
namely, the generalized eigenfunctions associated to $H$ also vanish at zero frequency,
as can be seen from 
\eqref{TRKodd} and \eqref{TRKeve}. 
This additional cancellation is responsible for stronger local decay (see Section \ref{DE}), 
or for the vanishing of the quadratic spectral distribution at zero frequency (see Section \ref{TQSD}). 
These effects are essentially equivalent; 
the former provides a more direct physical intuition, 
while the latter is the main tool which allows to simplify the proof.

In the more general framework of \cite{KGV1d}, these additional cancellations at zero frequency are absent, 
and this causes the formation of a singularity in frequency space:
the $L^2$ norm of $\partial_\xi \widetilde{f}(t,\xi)$ diverges rapidly as $t\to \infty$ 
close to the frequencies $\pm \sqrt{3}$, 
due to the resonant interaction $(0, 0 )\to \pm \sqrt{3}$
(i.e., when the frequency zero interacts with itself to give the frequency $\pm \sqrt{3}$).

The relevance of the weaker assumptions made in \cite{KGV1d}
comes from the fact that singularities in frequency space are a very general phenomenon for
many classes of $1$d problems. In particular, these singularities are expected to appear 
when the linear(ized) operator 
has a resonance at the bottom of the continuous spectrum 
(e.g. the linearized operators for the kinks of the $\phi^4$ and Sine-Gordon models) 
or a so-called internal mode (e.g. the linearized operator for the kink of the $\phi^4$ model);
see the discussion in Section 1 of \cite{KGV1d} for more on these aspects.


This type of singularity in frequency space requires a special functional framework,
like the one of \cite{KGV1d}, which makes the nonlinear estimates of $L^2$ weighted-type norms technical
and lengthy.
One of the goals of the present paper is to show how some of these estimates can be simplified 
in some more specific cases. 
At the same time we can also obtain better bounds by propagating a stronger, and more standard,
$L^2$ weighted-type norm. 

\smallskip
\subsection{Background} 
The equation \eqref{maineq} is motivated by the desire to understand 
the stability properties of topological solitons (kinks) in equations of the type
\begin{align}\label{wave}
\partial_t^2 \phi - \partial_x^2 \phi = W''(\phi),
\end{align}
where $W$ is a double-well potential.
A famous example is the kink $\tanh(x/\sqrt 2)$ in the $\phi^4$ model
\begin{align}\label{phi4}
\partial_t^2 \phi - \partial_x^2 \phi = \phi - \phi^3.
\end{align}
Local asymptotic stability for \eqref{phi4} and many interesting models of the form \eqref{wave}
was established recently by Kowalczyk, Martel and Mu\~noz \cite{KMM1}, 
see also \cite{KMM2,KMMVDB}, but some important questions remain to be understood, such as:
what is the global-in-space behavior of solutions? what is the rate of decay
of perturbations of solitons on long-time intervals? 
Can one establish global asymptotic stability?

The work of Delort and Masmoudi \cite{DM} gives a description of odd perturbations of the kink 
of \eqref{phi4} globally in space,
up to times of the order $\epsilon^{-4}$, where $\epsilon$ is the size of the perturbation. 
This limitation is due to a singularity at a frequency related to that of the internal mode.
Lindblad, Soffer and L\"uhrmann \cite{LLS20}, later joined by Schlag \cite{LLSS}, 
showed that, for non-generic $V$, solutions of \eqref{maineq} with $b=a_{\pm\infty}=0$
exhibit a logarithmic loss in the pointwise decay, 
because of the aforementioned singularity at the frequencies $\pm\sqrt{3}$.

We also mention \cite{GPR,ChPu}, 
which were earlier attempts to understand nonlinear resonances in the framework of the distorted Fourier transform,
and the recent work of Luhrman and Schlag \cite{LSSG} on the stability of the Sine-Gordon kink under odd perturbations,
and references therein.
All the articles mentioned above only correspond to very recent developments, 
and we refer to the introduction of \cite{KGV1d} and \cite{LSSG} for a more complete overview.

\smallskip
\subsection*{Notation}
Most of the necessary notation will be introduced in due course.
In Section \ref{sectionregular} and \ref{secred} we are going to use the following standard notation for cutoffs:
we fix a smooth even cutoff function $\varphi : \mathbb{R} \rightarrow [0, 1]$ 
supported in $[-8/5, 8/5]$ and equal to $1$ on $[-5/4, 5/4]$. For $k \in \mathbb{Z}$ we define $\varphi_k (\xi) := \varphi (2^{-k} \xi) -  \varphi (2^{-k+1} \xi) $, so that the family $(\varphi_k)_{k \in \mathbb{Z}}$ forms a partition of unity, $\sum_{k \in \mathbb{Z}} \varphi_k (\xi) =1$ for $\xi \neq 0$. Let 
\begin{equation}\label{LPnot}
\varphi_{I} (\xi) := \sum_{k \in I \cap \mathbb{Z}} \varphi_k (\xi), \quad
\varphi_{\leq a} (\xi) := \varphi_{(-\infty, a]} (\xi) , \quad
\varphi_{> a} (\xi) := \varphi_{(a, \infty)} (\xi) ,
\end{equation}
with similar definitions for $\varphi_{\geq a}  $, $\varphi_{ < a}$.
We will also denote $\varphi_{\sim k}$ a generic smooth cutoff function that 
is supported around $|\xi| \sim 2^k$, e.g. $\varphi_{[k-2, k+2]}$ or $\varphi_k'$.

We will denote by $T$ a positive time, and always work on an interval $[0,T]$
for our bootstrap estimates; see Proposition \ref{propbootstrap}.
To decompose the time integrals, such as those appearing in \eqref{bootQR}, 
for any $t \in [0,T]$, we will use a suitable decomposition of the indicator function $\mathbf{1}_{[0,t]}$
by fixing functions $\tau_0,\tau_1,\cdots, \tau_{L+1}: \R \to [0,1]$, 
for an integer $L$ with $|L-\log_2 (t+2)| < 2$,
with the properties that 
\begin{align}\label{timedecomp}
\begin{split}
& \sum_{n=0}^{L+1}\tau_n(s) = \mathbf{1}_{[0,t]},  
\qquad \supp (\tau_0) \subset [0,2], \quad  \supp (\tau_{L+1}) \subset [t/4,t],
\\
& \mbox{and} \quad \supp(\tau_n) \subseteq [2^{n-1},2^{n+1}], 
  \quad |\tau_n'(t)|\lesssim 2^{-n}, \quad \mbox{for} \quad n= 1,\dots, L.
\end{split}
\end{align}

\smallskip
\subsection{Organization of the article}

The emphasis in the present article is on the parts of the argument which can be simplified compared to \cite{KGV1d}; 
for estimates which are unchanged or that can be immediately adapted in the present setting,
we simply refer to the corresponding statements in \cite{KGV1d}.

A quick overview of the spectral theory of one-dimensional Schr\"odinger operators, 
and their associated distorted Fourier transform, is given in Section~\ref{STIDO}, 
followed by a discussion of estimates for the linear Klein-Gordon equation in Section~\ref{DE}. 
Structure theorems for the quadratic spectral distribution are stated in Section~\ref{TQSD}. 
For these three sections, the main novelty occurs in the case of even potentials with odd or even resonances, 
where our formulation is new, and allows to see more clearly the cancellation at frequency zero.

Section \ref{MND} recapitulates the normal form transformation, and the ensuing decomposition 
of nonlinear terms, referring to \cite{KGV1d} for further details. 
From there, the bootstrap argument is laid out in Section \ref{BABAPB}: 
Proposition \ref{propbootstrap} is the heart of the proof of Theorem \ref{mainthm}, 
and its proof occupies the remaining sections of the paper.

In the proof of Proposition \ref{propbootstrap}, the main difficulty is to estimate
the $L^2$-norm of $\partial_\xi \widetilde{f}$, and this is where the main simplifications occur.
The norm that we can propagate is stronger than the one in \cite{KGV1d},
and the proofs are shorter and simpler than those in Sections 8 and 9 of \cite{KGV1d}.
The (regular) quadratic term is treated in Section \ref{sectionregular}, 
and the (singular) cubic term is treated in Section \ref{TMSI}. 
In the final Section \ref{secred} we discuss how to estimate remainder terms 
and the other norms appearing in the bootstrap.

\smallskip
\subsection*{Acknowledgements}
While working on this project, PG was supported by the NSF grant DMS-1501019, 
by the Simons collaborative grant on weak turbulence, by the Center for Stability, Instability and Turbulence (NYUAD), 
and by the Erwin Schr\"odinger Institute.

FP was supported in part by a start-up grant from the University of Toronto, 
NSERC Grant No. 06487, and a Connaught Foundation New Researcher Award.

ZZ was supported by the Simons collaborative grant on weak turbulence and by an AMS-Simons travel grant.


\smallskip
\section{Spectral theory in dimension one}
This section starts with a quick review of the scattering 
and spectral theory of one-dimensional Schr\"odinger operators; 
the reader is referred to \cite{KGV1d} (see also references therein) for a more thorough presentation. 
Some helpful formulas in the case of even or odd function are then established.

\label{STIDO}
\subsection{Linear scattering theory}
Define $f_{+}(x,\xi)$ and $f_-(x,\xi)$ by the requirements that
\begin{align}
\label{f+-}
Hf_\pm := (- \partial_x^2 + V) f_{\pm}  = \xi^2 f_{\pm}, \quad \mbox{for all $x\in\R$, \quad and} \quad
\left\{
\begin{array}{ll}
f_{+}(x,\xi) \sim e^{ix\xi} & \mbox{as $x \to \infty$}
\\
f_{-}(x,\xi) \sim e^{-ix\xi} & \mbox{as $x \to - \infty$}.
\end{array}
\right.
\end{align}
Define further
\begin{align}
\label{m+-}
m_{+}(x,\xi) = e^{-i\xi x} f_{+}(x,\xi) \quad \mbox{and} \quad m_{-}(x,\xi) = e^{i\xi x} f_{-}(x,\xi),
\end{align}
so that $m_{\pm}$ is a solution of 
\begin{equation}
\label{equationm}
\partial_x^2 m_{\pm} \pm 2i\xi \partial_x m_{\pm} = Vm_{\pm}, \qquad m_\pm(x,\xi) \to 1 \;\mbox{as $x \to \pm \infty$}.
\end{equation}
These satisfy the estimates
\begin{equation}
\label{estimatesm}
\begin{array}{ll}
\left| \partial_x^\alpha \partial_\xi^\beta (m_{\pm}(x,\xi) - 1) \right| 
  \lesssim \langle x \rangle^{-N} \langle \xi \rangle^{-1-\beta} & \mbox{if $\pm x > -1$},
\\
 \left| \partial_x^\alpha \partial_\xi^\beta m_{\pm}(x,\xi) \right| 
 \lesssim \langle x \rangle^{1+\beta} \langle \xi \rangle^{-1-\beta} & \mbox{if $\pm x < 1$}.
\end{array}
\end{equation}
Denote $T(\xi)$ and $R_{\pm}(\xi)$ respectively the {\it transmission} 
and {\it reflection} coefficients associated to the potential $V$.
These coefficients are such that
\begin{align}
\label{f+f-}
\begin{split}
&f_+ (x,\xi) = \frac{1}{T(\xi)} f_-(x,-\xi) + \frac{R_-(\xi)}{T(\xi)} f_-(x,\xi),
\\
&f_- (x,\xi) = \frac{1}{T(\xi)} f_+(x,-\xi) + \frac{R_+(\xi)}{T(\xi)} f_+(x,\xi),
\end{split}
\end{align}
and are explicitly given by the formulas
\begin{align}
\label{TRformula}
\begin{split}
&T(\xi) = \frac{2i\xi}{2i\xi - \int V(x) m_{\pm} (x,\xi)\,dx},
\\
& R_{\pm}(\xi) = \frac{\int e^{\mp 2i\xi x} V(x) m_{\mp}(x,\xi)\,dx}{2i\xi - \int V(x) m_{\pm} (x,\xi)\,dx}.
\end{split}
\end{align}
These formulas are only valid for $\xi \neq 0$ a priori, 
but  $T$ and $R_{\pm}$ can be extended to be smooth functions on the whole real line. 
From \eqref{TRformula} and \eqref{estimatesm} we can see that
\begin{equation}
\label{estimatesTR}
\begin{split}
& |\partial_\xi^\beta(T(\xi) - 1) | \lesssim  \langle \xi \rangle^{-1-\beta},
\qquad |\partial_\xi^\beta R_{\pm}(\xi) | \lesssim  \langle \xi \rangle^{-N}.
\end{split}
\end{equation}
Finally, the scattering matrix
$$
S(\xi) = \begin{pmatrix} T(\xi) & R_+(\xi) \\ R_-(\xi) & T(\xi) \end{pmatrix}
$$
is unitary.

\smallskip
\subsection{Exceptional and generic potentials}
Here are some useful lemmas about the behavior of the transmission and reflection coefficients. Readers can refer to Deift-Zhou \cite{DZ1} for details. Recall the definition of the Wronskian $W(f,g) = f'g - fg'$.

\begin{df}\label{defV}
We call the potential $V$
\begin{itemize}
\item \emph{generic} if $\int V(x) m_\pm (x, 0) dx \neq 0$; 
\item \emph{exceptional} if $\int V(x) m_\pm (x, 0) dx = 0$. 
\end{itemize}
\end{df}

\begin{lem}\label{lemVgen}
The four following assertions are equivalent

\begin{itemize}
\item[(i)] 
$V$ is generic.

\smallskip
\item[(ii)] $\displaystyle T(0) = 0, R_{\pm}(0) = -1$.

\smallskip
\item[(iii)] The Wronskian of $f_+$ and $f_-$ at $\xi=0$ is non-zero.

\smallskip
\item[(iv)] The potential $V$ does not have a resonance at $\xi=0$, 
in other words there does not exist a bounded non trivial solution in the kernel of $-\partial_x^2 + V$.

\end{itemize}

\end{lem}

\begin{prop}[Low energy scattering]\label{LES}
If $V$ is generic, there exists $\alpha \in i \mathbb{R}$ such that
\begin{align}
\label{LEST}
T(\xi) = \alpha \xi + O(\xi^2).
\end{align}
If $V$ is exceptional, let
$$
a := f_+(-\infty,0) \in \mathbb{R} \setminus \{ 0 \}.
$$
Then, 
\begin{align}\label{LESTR0}
T(0) = \frac{2a}{1+a^2}, \qquad R_+(0) = \frac{1 - a^2}{1+ a^2}, \qquad \mbox{and} \qquad R_-(0) = \frac{a^2-1}{1+ a^2}.
\end{align}
\end{prop}

\begin{lem}\label{estimatesTR'}
Assume that $V$ is even and exceptional, with an even zero energy resonance. Then
\begin{align}\label{estimatesTRodd}
 \jxi^{1+\beta} \Big| \partial_\xi^\beta \frac{T(\xi) - 1}{\xi} \Big|
 +  \jxi^{N} \Big| \partial_\xi^\beta \frac{R_\pm(\xi)}{\xi}\Big| \lesssim 1.
\end{align}
Similarly, if the zero energy resonance is odd we have
\begin{align}\label{estimatesTReven}
 \jxi^{1+\beta} \Big| \partial_\xi^\beta \frac{T(\xi) + 1}{\xi} \Big|
 +  \jxi^{N} \Big| \partial_\xi^\beta \frac{R_\pm(\xi)}{\xi}\Big| \lesssim 1.
\end{align}
\end{lem}

\begin{proof}
These properties can be verified starting from the formulas \eqref{TRformula} and making use of Proposition \ref{LES}. 
From Proposition \ref{LES}, we know that 
when the zero energy resonance $f_+(x,0)$ is even we have $T(0) = 1$ since $a=1$ 
(due to the fact that $f_+(+\infty,0) = 1$ and $a = f_+(-\infty,0) $),  
so $T(\xi)-1 $ is of $O(\xi)$, and the estimate for $T(\xi)$ in \eqref{estimatesTRodd} 
follows from the smoothness of $T$ and the
decay in \eqref{estimatesm}. 
Similarly, for the case when the zero energy resonance is odd, we have $T(0) = -1$ 
since $a=-1$, and hence $T(\xi)+1$ is of $O(\xi)$ 
and \eqref{estimatesTReven} follows. 
The estimates for $R_\pm (\xi)$ can be verified in a similar way using that $R_\pm (0)= 0$ 
when $a=\pm 1$. 
\end{proof}

\medskip
\subsection{Flat and distorted Fourier transform}

The normalization we adopt for the flat Fourier transform is
$$
\widehat{\mathcal{F}} \phi (\xi) = \widehat{\phi}(\xi) = \frac{1}{\sqrt{2\pi}} \int e^{-i\xi x} \phi(x) \, dx.
$$
Its inverse is given by
$$
\widehat{\mathcal{F}}^{-1} \phi = \frac{1}{\sqrt{2\pi}} \int e^{i\xi x} \phi(\xi) \, d\xi.
$$
We now define the wave functions associated to $H$:
\begin{align}
 \label{psixk}
\psi(x,\xi) := \frac{1}{\sqrt{2\pi}}
\left\{
\begin{array}{ll}
T(\xi) f_+(x,\xi) & \mbox{for $\xi \geq 0$} \\ \\
T(-\xi) f_-(x,-\xi) & \mbox{for $\xi < 0$}.
\end{array}
\right.
\end{align}
The distorted Fourier transform is then defined by
\begin{align}
\label{distF}
\widetilde{\mathcal{F}} \phi(\xi) = \widetilde {\phi}(\xi) = \int_\mathbb{R} \overline{\psi(x,\xi)} \phi(x)\,dx.
\end{align}
It is self-adjoint and has the inverse
\begin{align}
\widetilde{\mathcal{F}}^{-1} \phi(x) = \int \psi(x,\xi) \phi(\xi) \,d\xi.
\end{align}

Let $\rho$ be an even, smooth, non-negative function, equal to $0$ outside of $B(0,2)$ and such that $\int \rho = 1$. 
Define $\chi_{\pm}$ by
\begin{equation}
\label{chi+-}
\chi_+(x) = H * \rho = \int_{-\infty}^{x} \rho(y)\,dy , \quad \mbox{and} \quad \chi_+(x) + \chi_-(x) = 1,
\end{equation}
where $H$ is the Heaviside function, $H=\mathbf{1}_{+} = \mathbf{1}_{[0,\infty)}$. Notice that
$$
\chi_+(x) = \chi_-(-x).
$$



With $\chi_\pm$ as above, and using the definition of $\psi$ in \eqref{psixk}
and $f_\pm$ and $m_\pm$ in \eqref{f+-}-\eqref{m+-}, as well as the identity \eqref{f+f-} we can write
\begin{align}
\label{psi>}
\begin{split}
\mbox{for} \quad \xi>0 \qquad \sqrt{2\pi}\psi(x,\xi) & = \chi_+(x)T(\xi)m_+(x,\xi)e^{ix\xi}
  \\ & + \chi_-(x) \big[ m_-(x,-\xi)e^{i\xi x} + R_-(\xi) m_-(x,\xi)e^{-i\xi x} \big],
\end{split}
\end{align}
and
\begin{align}
\label{psi<}
\begin{split}
\mbox{for} \quad \xi<0 \qquad \sqrt{2\pi}\psi(x,\xi) & = \chi_-(x)T(-\xi)m_-(x,-\xi)e^{ix\xi}
  \\ & + \chi_+(x) \big[m_+(x,\xi) e^{i\xi x} + R_+(-\xi)m_+(x,-\xi) e^{-i\xi x} \big].
\end{split}
\end{align}

Below we give some general properties of the distorted Fourier transform.
The proofs of these statements can be found, for example, in Proposition 3.6 of \cite{KGV1d},
and more details can be found in \cite{Yafaev}.

\begin{prop}[Mapping properties of the distorted Fourier transform]\label{propFT}
With $\wtF$ defined in \eqref{distF},

\begin{itemize}
\smallskip
\item[(i)] $\widetilde{\mathcal{F}}$ is a unitary operator from $L^2$ onto $L^2$. In particular, its inverse is
$$
\widetilde{\mathcal{F}}^{-1} \phi(x) = {\widetilde{\mathcal{F}}}^\ast \phi(x) = \int_\R \psi(x,\xi) \phi(\xi)\,d\xi.
$$

\smallskip
\item[(ii)] $\widetilde{\mathcal{F}}$ maps $L^1(\mathbb{R})$ to functions in $L^\infty(\mathbb{R})$ 
which are continuous at every point except $0$, and converge to $0$ at $\pm \infty$. 

\smallskip
\item[(iii)] $\widetilde{\mathcal{F}}$ maps the Sobolev space $H^s(\mathbb{R})$ 
onto the weighted space $L^2(\langle \xi \rangle^{2s} \,d\xi)$.

\smallskip
\item[(iv)] If $\widetilde{f}$ is continuous at zero, then, for any integer $s \geq 0$
\begin{align*}
\|\langle \xi \rangle^s \partial_\xi \widetilde{f} \|_{L^2} \lesssim \| f \|_{H^s} + \| \langle x \rangle f \|_{H^s}.
\end{align*}

\end{itemize}
\end{prop}

\medskip
\subsection{The wave operator}

The wave operator associated to $H$ is defined as
\begin{align}
\mathcal{W} = s-\lim_{t \rightarrow \infty} e^{itH} e^{-it H_0} . 
\end{align}
It has the following property, see Section 3.2.3 in \cite{KGV1d}:
\begin{prop}
The wave operator is unitary on $L^2$ and satisfies
\begin{align}
\mathcal{W} = \widetilde{\mathcal{F}}^{-1} \widehat{\mathcal{F}},
\qquad \mathcal{W}^{-1} = \mathcal{W}^\ast =  \widehat{\mathcal{F}}^{-1}\widetilde{\mathcal{F}} , 
\end{align}
and it intertwines $H$ and $H_0$: 
\begin{align}
f(H) = \mathcal{W}  f(H_0) \mathcal{W}^\ast . 
\end{align}
Moreover, $\mathcal{W}$ and $\mathcal{W}^\ast$ extend to bounded operators on $W^{k,p} (\mathbb{R})$ for any $k$ and $1 < p < \infty$. Furthermore, in the exceptional case, if $f_+(-\infty,0) = 1$, this remains true if $p = 1$ or $\infty$.
\end{prop}


\medskip
\subsection{The case of even and odd functions}

\begin{lem}[Parity preservation]\label{lemevenodd0}
Let the potential $V$ be even.
Then, we have 
\begin{align}\label{lemevenoddf+-}
f_+(x,\xi) =  f_-(-x,\xi), \qquad R_+(\xi) = R_-(\xi) \equiv R(\xi), \qquad \psi(x,\xi) = \psi(-x,-\xi),
\end{align}
and the distorted Fourier transform preserve evenness and oddness.
\end{lem}

\begin{proof}
The first identity in \eqref{lemevenoddf+-} is a consequence of the uniqueness of solutions for the ODE \eqref{f+-},
$R_+=R_-$ follows directly from \eqref{TRformula} and $m_+(x,\xi) = m_-(-x,\xi)$ (see \eqref{m+-}).
The identity for $\psi$ follows from the previous ones and the definition \eqref{psixk}.
The preservation of parity for the distorted Fourier transform then follows from \eqref{distF}.
\end{proof}

\begin{lem}[Formulas with parity]\label{lemevenodd}
Let the potential $V$ be even and exceptional. 


\setlength{\leftmargini}{1.5em}
\begin{itemize}

\smallskip
\item When $f$ is odd, and the zero energy resonance is even, we have the formulas
\begin{align}\label{invFodd}
\begin{split}
\wt{f}(\xi) & =  \int_{\R} \bar{\psi_o(x,\xi)} f(x) \, dx,
\\
f(x) & =  \int_{\R} \psi_o(x,\xi) \wt{f}(\xi) \, d\xi,
\end{split}
\end{align}
where 
\begin{align}\label{Kodd}
\begin{split}
& \psi_o(x,\xi) :=  \mathbf{1}_+(\xi) \chi_+(x) K_o(x,\xi) + \mathbf{1}_-(\xi) \chi_-(x) K_o(-x,-\xi),
 \\
& \sqrt{2\pi} K_o(x,\xi) := T(\xi)m_+(x,\xi)e^{ix\xi} - m_+(x,-\xi) e^{-ix\xi} - R(\xi) m_+(x,\xi)e^{i\xi x},
\end{split}
\end{align}
We also have
\begin{align}\label{TRKodd}
K_o(x,0)=0, 
\end{align}
and can write
\begin{align}\label{invFodd'}
\begin{split}
f(x) & =  \int_0^\infty \big[ \chi_+(x)K_o(x,\xi) - \chi_-(x)K_o(-x,\xi) \big] \wt{f}(\xi) \, d\xi.
\end{split}
\end{align}


\smallskip
\item When $f$ is even, and the zero energy resonance is odd, we have the formulas
\begin{align}\label{invFeve}
\begin{split}
& \wt{f}(\xi) =  \int_{\R_+} \bar{\psi_e(x,\xi)} f(x) \, dx, 
\\
& f(x) =  \int_{\R_+} \psi_e(x,\xi) \wt{f}(\xi) \, d\xi, 
\end{split}
\end{align}
where
\begin{align}\label{Keven}
\begin{split}
& \psi_e(x,\xi) := \mathbf{1}_+(\xi) \chi_+(x) K_e(x,\xi) + \mathbf{1}_-(\xi) \chi_-(x) K_e(-x,-\xi),
\\
& \sqrt{2\pi} K_e(x,\xi) := 
   T(\xi)m_+(x,\xi)e^{ix\xi} +m_+(x,-\xi) e^{-ix\xi} 
  + R(\xi) m_+(x,\xi)e^{i\xi x} .  
\end{split}
\end{align}
We also have
\begin{align}\label{TRKeve}
\quad K_e(x,0)=0. 
\end{align}
and can write
\begin{align}\label{invFeve'}
\begin{split}
f(x) & =  \int_0^\infty \big[ \chi_+(x)K_e(x,\xi) + \chi_-(x)K_e(-x,\xi) \big] \wt{f}(\xi) \, d\xi.
\end{split}
\end{align}


\end{itemize}

\end{lem}

Note that there is a slight asymmetry between the cutoffs in $x$ and $\xi$ since we are using smooth cutoffs $\chi_\pm$ in the $x$ variable.

\begin{proof}
The formulas \eqref{invFodd}-\eqref{Kodd} and \eqref{invFeve}-\eqref{Keven} 
can be directly verified starting from \eqref{distF} and using \eqref{psi>}-\eqref{psi<},
as well as $R_+=R_-$.
We show this in the case when $f$ is odd. We have, for $\xi >0$, and denoting $\bar{z} = z^\ast$,
\begin{align*}
\wt{f}(\xi) & = \int_\R \big[ \chi_+(x) T(\xi)m_+(x,\xi)e^{ix\xi} 
  + \chi_-(x) (m_-(x,-\xi)e^{i\xi x} + R
  (\xi) m_-(x,\xi)e^{-i\xi x}) \big]^\ast \, f(x) \, dx
  \\
  & = \int_\R \chi_+(x) \big[ T(\xi)m_+(x,\xi)e^{ix\xi} 
  - m_-(-x,-\xi)e^{-i\xi x} - R(\xi) m_-(-x,\xi)e^{i\xi x} \big]^\ast \, f(x) \, dx
  \\
  & = \int_\R \chi_+(x) \bar{K_o(x,\xi)} \, f(x) \, dx.
\end{align*}
Using also that $m_-(-x,\xi) = m_+(x,\xi)$, and a similar identity for $\xi<0$, 
gives us the first formula in \eqref{invFodd}.
The property \eqref{TRKodd} follows from Proposition \ref{LES} since $T(0) = 1$ and $R(0) = 0$ 
under our parity assumptions ($a=1$ when the resonance is even).

\eqref{invFeve}-\eqref{TRKeve} follow similarly by explicit computations, 
and using again Proposition \ref{LES} to see that
$a=-1$ when the resonance is odd and therefore $T(0) = -1$ and $R(0) = 0$.
\end{proof}


\smallskip
\section{Linear decay estimates}
\label{DE}
Let 
\begin{align}
B := \sqrt{-\partial_{xx} + V + 1}.
\end{align}
We have the following standard decay estimate:

\begin{proposition}[Pointwise decay]\label{propdecay0}
Assume that either $V$ is generic, or that $V$ is exceptional with an
even, resp. odd, zero energy resonance and that $f$ is odd, resp. even.
Then, we have
\begin{align}\label{dec} 
{\|e^{itB} P_c f \|}_{L^\infty} \lesssim \jt^{-1/2} {\| \jxi^{3/2} \wt{f} \|}_{L^\infty} + 
  \jt^{-11/20} {\| \jxi \partial_\xi \wt{f} \|}_{L^2} + \jt^{-7/12} {\| f \|}_{H^4}.
\end{align}
\end{proposition}

This statement corresponds to Proposition 3.11 in \cite{KGV1d}, 
where the factor of $\jt^{-11/20}$, which is not optimal, 
is replacing the factor of $\jt^{-3/4 + \beta \gamma}$
(where $\beta,\gamma$ were chosen so that $\beta\gamma = 1/4 -$);
this latter was present in the cited work due to the ``bad'' frequencies $\pm \sqrt{3}$
that required a special treatment in \cite{KGV1d}, which is not the case in the present article.
We also refer the reader to Lemma 2.2 of \cite{LLS20} where a linear estimate
similar to \eqref{dec} is proved (for the case $V=0$); the slightly different rate 
of decay in front of the weighted $L^2$ norm is due to the different handling
of the high frequencies.

\medskip
Next, we provide an improved local decay estimate.
It will not be used in this work, but we include it for the sake of completeness and future reference.
Such improved decay is one of the keys to nonlinear stability, 
but it will be more convenient for us to take another route to proving nonlinear bounds, 
working in (distorted) Fourier space, rather than using physical space estimates.


\begin{proposition}[Local decay]\label{proplocdec}
Assume that either $V$ is generic, or that $V$ is exceptional with an
even, resp. odd, zero energy resonance and that $f$ is odd, resp. even.
Then, we have
\begin{align}\label{locdec1}
\begin{split}
{\big\| \jx^{-2} e^{itB} P_c f \big\|}_{L^\infty_x} 
  \lesssim \jt^{-1} \Big( {\big\| \partial_\xi \wt{f} \big\|}_{L^1_\xi}  
  + {\big\| \wt{f} \big\|}_{L^1_\xi} + {\| f \|}_{H^1_x} \Big). 
\end{split}
\end{align}
\end{proposition}


\begin{rem}
One should notice that, in the more general case treated in \cite{KGV1d}, 
or in the exceptional case without the present symmetry assumptions, 
decay rates as strong as \eqref{locdec1} are not available.
\end{rem}

\begin{proof}[Proof of Proposition \ref{proplocdec}] 
We are first going to give 
the details of the proof in the case of an exceptional potential with even resonance
and with $f$ odd, and we will then indicate how the other cases can be treated similarly.

\smallskip
\noindent
Let us bound the first quantity on the left-hand side of \eqref{locdec1}. 
In view of Sobolev's embedding, we may assume $t \geq 1$.
We look at different cases depending on the parity.

\medskip
\noindent
{\it The case of odd $f$}.
By \eqref{invFodd}-\eqref{invFodd'}, we can write:
\begin{align*}
\big(e^{i t B }f\big)(x) = \chi_+(x) 
 \int_{\R_+} e^{ it\jxi} K_o(x,\xi)  \wt{f}(\xi) \, d\xi 
 - \chi_-(x)\int_{\R_+} e^{it\langle \xi \rangle} K_o(-x,\xi)  \wt{f}(\xi) \, d\xi. 
\end{align*}
We will only bound the first term above, since the other one can be treated identically.
Observe that we can write \eqref{Kodd} as
\begin{align*}
\sqrt{2\pi} K_o(x,\xi) = m_+(x,-\xi) 2i \operatorname{sin}(x\xi) + [m_+(x,\xi) - m_+(x,-\xi)] e^{ix\xi}
  \\ + (T(\xi) - 1 - R(\xi))m_+(x,\xi) e^{ix\xi}.
\end{align*}
This expression makes the cancellations even more apparent, leading, 
with the help of \eqref{estimatesm} and \eqref{estimatesTRodd}, to the estimates
\begin{align}\label{Ko+bounds}
\begin{split}
& |K_o(x,\xi)| \lesssim \min( |x\xi| , 1) + \min(\xi,1), \qquad |\partial_\xi K_o(x,\xi)| \lesssim \jx,
\\
& \big|\partial_\xi \big(\xi^{-1} K_o(x,\xi)\big) \big| \lesssim \jx^2 , \qquad |\xi|\leq 1,
\end{split}
\end{align}
valid in the support of $\chi_+(x)$.
Using that $K_o(x,0) = 0 = \wt{f}(0)$, 
we can integrate by parts through the identity $e^{it\jxi} = \frac{\jxi}{i t\xi} \partial_\xi e^{it\jxi}$ to obtain
\begin{align*}
& \frac{\chi_+(x)}{\jx^2} \left| \int_{\R_+} e^{it\jxi} K_o(x,\xi) \wt{f}(\xi) \, d\xi \right| 
  \\
& \qquad \lesssim \frac{1}{t \jx^2} \int \frac{\jxi}{|\xi|} |K_o(x,\xi)| |\partial_\xi \wt{f}(\xi)| \,d\xi 
  + \frac{1}{t \jx^2} \int |\wt{f}(\xi)| 
  \left| \partial_\xi \left[\frac{\jxi}{\xi} K_o(x,\xi) \right] \right| \,d\xi.
\end{align*}
Using the bounds \eqref{Ko+bounds}
this is
\begin{align*}
\lesssim \frac{1}{t \jx^2} \int \frac{\jxi}{|\xi|} 
  (\min(|x\xi|,1) + \min(\xi,1)) |\partial_\xi \wt{f}(\xi)| \,d\xi 
  + \frac{1}{t} \int |\wt{f}(\xi)| \,d\xi 
  \\
  \lesssim \frac{1}{t} \left[ {\|  \partial_\xi \wt{f} \|}_{L^1}
  + {\| \wt{f} \|}_{L^1}  \right].
\end{align*}

\medskip
\noindent
{\it The case of even $f$}.
When $f$ is even we can use \eqref{invFeve}-\eqref{invFeve'} to write
\begin{align*}
e^{i t B } f  = \chi_+(x) \int_{\R_+} e^{ i t \jxi} K_e(x,\xi) \wt{f}(\xi) \, d\xi 
  +  \chi_-(x) \int_{\R_+} e^{ i t \langle \xi \rangle} K_e(-x,\xi)  \wt{f}(\xi) \, d\xi.
 \end{align*}  
Similarly to the odd case, also here we have cancellation at zero 
(see \eqref{TRKeve}).
It suffices again to just look at the first integral. Writing
\begin{align*}
\sqrt{2\pi} K_e(x,\xi) 
  = -m_+(x,\xi) 2i \sin(x\xi) + [ m_+(x,-\xi) - m_+(x,\xi)] e^{-ix\xi}
  \\ 
  + (T(\xi) + 1 + R(\xi))m_+(x,\xi) e^{ix\xi},
\end{align*}
we see that estimates like \eqref{Ko+bounds} hold for $K_e$ as well,
and we can apply the same argument above.

\medskip
\noindent
{\it The case of generic $V$}.
When $V$ is generic we have $T(0)=0$ and $R_\pm(0)=-1$
and arguments similar to those above can be applied to the general formula \eqref{distF}
with \eqref{psi>}-\eqref{psi<}.
%
%
\end{proof}


\smallskip
\section{The quadratic spectral distribution}\label{TQSD}
In this section we analyze the quadratic (and cubic) spectral distribution
and will mainly focus on the analysis that is relevant to the case of odd solutions 
with an even resonance.
At the end of the section we will also indicate how to analyze similarly
the case of even solutions with odd resonances, 
and the case of generic $V$.
We denote, for $f$ a complex-valued function,
$$
f_+ = f, \qquad f_- = \overline{f}.
$$



\subsubsection*{The odd case}
Recall that in the case of odd data (and solution) we assume that the zero energy resonance is even.
Our starting point are the formulas \eqref{invFodd}-\eqref{invFodd'}.

\begin{proposition}[The odd case]\label{propmuodd}
Let $f,g \in \mathcal{S}$ be odd functions, and $a=a(x)$ be odd and satisfying \eqref{ab}.
There exists a distribution $\mu^{o}_{\iota_1\iota_2}$, $\iota_1,\iota_2 \in \{+,-\}$,
such that, for $\xi \geq 0$,
\begin{align}
\label{propmu2}
\widetilde{\mathcal{F}} \big(a \, f_{\iota_1} g_{\iota_2} \big)(\xi) = \iint_{(\R_+)^2} 
  (\wt{f})_{\iota_1} (\eta) (\wt{g})_{\iota_2} (\sigma)\,\mu^o_{\iota_1 \iota_2}(\xi,\eta,\sigma) \, d\eta \,d\sigma,
\end{align}
with odd extension to $\xi < 0$,
and such that $\mu^o_{\iota_1\iota_2}$ can be split into a singular and a regular part as follows:
\begin{align*}
(2\pi) \mu^o_{\iota_1\iota_2} = \mu^{o,S}_{\iota_1\iota_2} + \mu^{o,R}_{\iota_1\iota_2},
\end{align*}
where:

\begin{itemize}

\item The singular part $\mu^{o,S}_{\iota_1\iota_2}$ is given by
\begin{align}\label{propmuS}
\begin{split}
\mu^{o,S}_{\iota_1 \iota_2}(\xi,\eta,\sigma) & = \sum_{\lambda, \mu, \nu \in \{+,-\}} 
\bar{a_{\lambda} (\xi)} (a_{\mu}(\eta))_{\iota_1} (a_{\nu} (\s))_{\iota_2}
  \\ & \times \ell_{+\infty} \left[ \sqrt{\frac{\pi}{2} }\delta(p) + 
  \,i \, \varphi^{\ast}(p,\eta,\s)\, \pv \frac{\what{\phi}(p)}{p} \right]_{p = \lambda \xi - \mu \iota_1 \eta - \nu \iota_2 \sigma}
\end{split}
\end{align}
where the coefficients are given as
\begin{align}\label{propmuScoeff}
\begin{split}
a_{+} (\xi) = T(\xi)-R(\xi), \qquad a_{-} (\xi) = -1,
\end{split}
\end{align}
$\phi \in \mathcal{S}$ is even with integral one, and 
\begin{align}\label{propmustar}
\varphi^\ast(p,\eta,\sigma) = \varphi_{\leq -D}\big(p R(\eta,\sigma)\big),
\quad R(\eta,\sigma) = \frac{\langle \eta \rangle \langle \sigma \rangle}{\langle \eta \rangle + \langle \sigma \rangle}.
\end{align}


\medskip
\item The regular part $\mu^{o,R}_{\iota_1\iota_2}$ satisfies, for $\xi,\eta,\s >0$
\begin{equation}
\label{estmu1}
\mu^{o,R}_{\iota_1 \iota_2}(\xi,\eta,\sigma) 
  = \frac{\xi \cdot \eta  \cdot \s}{\jxi\jeta\jsig} \mathfrak{q}^{o}_{\iota_1 \iota_2}(\xi,\eta,\sigma)
\end{equation}
where 
\begin{equation}
\label{estmu2}
| \partial_\xi^a \partial_\eta^b \partial_\sigma^c \mathfrak{q}^{o}_{\iota_1 \iota_2}(\xi,\eta,\sigma) | 
  \lesssim \sup_{\mu,\nu} \frac{1}{\langle \xi +\mu \eta +\nu \sigma \rangle^N} |R(\eta,\s)|^{1+a+b+c},
  \qquad |a|+|b|+|c| \leq N.
\end{equation}
\end{itemize}

\end{proposition}

\begin{rem}[About the formula \eqref{propmuS}]
It is important to note that the formula \eqref{propmuS} gives a measure which vanishes when one of the three
frequencies $\xi,\eta$ or $\s$ is zero.
To see this  when $\eta = 0$, say, we can fix $\iota_1=\iota_2=+$ since they are irrelevant,
and recall that $a_+(0)-1=0$, so that $a_\mu(0) = \mu$ and we have
\begin{align}\label{muSat0}
\begin{split}
\mu^{o,S}_{++}(\xi,0,\sigma) 
& = \sum_\mu \mu \sum_{\lambda, \nu}
\bar{a_{\lambda} (\xi)} \cdot a_{\nu}(\s)
  \\
  & \times \ell_{+\infty} \left[ \sqrt{\frac{\pi}{2} }\delta(p_0) + 
  \,i \, \varphi^{\ast}(p_0R(0,\s))\, \pv\frac{\what{\phi}(p_0)}{p_0} \right]_{p_0 = \lambda \xi - \nu \sigma} = 0
\end{split}
\end{align}
The same argument works symmetrically when $\s=0$, and similarly when $\xi=0$.
\end{rem}

\begin{proof}[Proof of Proposition \ref{propmuodd}]
In what follows we let $\iota_1=\iota_2=+$ for simplicity 
and also omit the superscripts $o$ from the formulas, so that, for example,
$\mu$ will stand for $\mu^o_{\iota_1\iota_2}$.
Since $a f g(x)$ is odd, from the formulas \eqref{invFodd} and \eqref{invFodd'},
we formally have \eqref{propmu2} by defining, for $\xi > 0$
\begin{align}\label{propmu1}
\begin{split}
\mu(\xi,\eta,\sigma) 
= \int_{\R} a(x) \chi_+(x) \overline{K_o(x,\xi)} 
  \, & \big[ \chi_+(x)K_o(x,\eta) - \chi_-(x)K_o(-x,\eta) \big] 
  \\ \times & \big[ \chi_+(x)K_o(x,\s) - \chi_-(x)K_o(-x,\s) \big] \,dx,
\end{split}
\end{align}
and for $\xi<0$
\begin{align}\label{propmu1'}
\begin{split}
\mu(\xi,\eta,\sigma) 
= \int_{\R} a(x) \chi_-(x) \overline{K_o(-x,-\xi)} 
  \, & \big[ \chi_+(x)K_o(x,\eta) - \chi_-(x)K_o(-x,\eta) \big] 
  \\ \times & \big[ \chi_+(x)K_o(x,\s) - \chi_-(x)K_o(-x,\s) \big] \,dx
  \\
  & = - \mu(-\xi,\eta,\sigma).
\end{split}
\end{align}
Note that the limit at $\xi=0$ exists and is zero. 
The limit as $\eta,\s \rightarrow 0^+$ is also zero.

Recall the formula \eqref{Kodd} and decompose $K_o(x,\xi)$ into a singular and regular part:
\begin{align}\label{Kosplit}
\begin{split}
& K_o(x,\xi) =  K_o^{S} (x,\xi) + K_o^{R} (x,\xi),
\\
& \sqrt{2\pi}K_o^{S}(x,\xi) = T(\xi) e^{ix\xi} - e^{-ix\xi}  - R(\xi) e^{ix\xi} 
  = \sum_{\l\in\{+,-\}} a_\l(\xi)e^{i \l x\xi},
\\
& \sqrt{2\pi}K_o^{R}(x,\xi) = (T(\xi)-R(\xi))(m_+(x,\xi)-1)e^{ix\xi} - (m_+(x,-\xi)-1)e^{-ix\xi}.
\end{split}
\end{align}
Note that, using \eqref{TRKodd}, 
\eqref{estimatesTR}, and \eqref{estimatesm}, 
on the support of $\chi_+(x)$ we have (see also \eqref{Ko+bounds})
\begin{align}
\label{mupropK1}
& |K_o(x,\xi)| 
+ |K_o^{S}(x,\xi)|   \lesssim \min(|\xi|,1) + \min (|x\xi|,1),
\\
\label{mupropK2}
& |K_o^{R}(x,\xi)| \lesssim 
  \jx^{-N+1} \min(|\xi|,1).
\end{align}

Using the above decomposition we can write, for $\xi\geq 0$,
\begin{align}\label{muprop5}
\begin{split}
\mu(\xi,\eta,\sigma) 
& = 
  \int_{\R} a(x) (\chi_+(x))^3 \overline{K^{S}_o(x,\xi)} K^{S}_{o}(x,\eta) K^{S}_{o}(x,\sigma) \,dx
  + \mu_{R,0}(\xi,\eta,\s),
\end{split}
\end{align}
where $\mu_{R,0}$ is a remainder term that will be absorbed in $\mu_R$. 
For the leading order we first write $a\chi_+^3 = (a - \ell_{+\infty})\chi_+^3 + \ell_{+\infty}\chi_+^3$,
and note that the first function 
decays as fast as $a-\ell_{+\infty}$;
then, we use the fact that 
\begin{align}\label{muprop5.1} 
\widehat{(\chi_+)^3}(\xi)
= \sqrt{\frac{\pi}{2}} \delta(\xi) +  \pv \frac{\widehat{\phi}(\xi)}{i\xi}  + \widehat{\psi}(\xi),
\end{align}
for some even, smooth, compactly supported functions $\phi,\psi$ with $\phi$ having integral $1$.
The 
formula for $\mu^S$ in \eqref{propmuS} with coefficients as in \eqref{propmuScoeff} 
then comes from the first two terms in \eqref{muprop5.1}, that is,
\begin{align}\label{muprop5.5}
\begin{split}
& \ell_{+\infty} \int_{\R} \big( \chi_+^3(x) - \psi(x) \big)
  \overline{K^{S}_o(x,\xi)} K^{S}_{o}(x,\eta) K^{S}_{o}(x,\sigma) \,dx,
\end{split}
\end{align}
by using the expression for $K_o^S(x,\xi)$ in \eqref{Kosplit} 
and \eqref{muprop5.1}, and inserting in addition the cutoff $\varphi^{\ast}$ in front of the $\pv$ term.
We are then left with three contributions from \eqref{muprop5}, besides $\mu_{R,0}$, that is,
\begin{align}\label{mupropR1}
\mu_{R,1}(\xi,\eta,\sigma) & = 
  \int_{\R} (a(x) - \ell_{+\infty}) (\chi_+(x))^3 \overline{K^{S}_o(x,\xi)} K^{S}_{o}(x,\eta) K^{S}_{o}(x,\sigma) \,dx,
  \\
\label{mupropR2}
\mu_{R,2}(\xi,\eta,\sigma) & = 
  \int_{\R} \ell_{+\infty} \psi(x) \overline{K^{S}_o(x,\xi)} K^{S}_{o}(x,\eta) K^{S}_{o}(x,\sigma) \,dx, 
\\
\label{mupropR3}
\mu_{R,3}(\xi,\eta,\sigma) & = \sum_{\lambda, \mu, \nu}
  \bar{a_{\lambda} (\xi)} a_{\mu}(\eta) a_{\nu} (\s) 
  \ell_{+\infty}
  i \big[1-\varphi^{\ast}(p,\eta,\s)\big] \, \frac{\what{\phi}(p)}{p} \, \Big|_{
  p = \lambda \xi - \mu \iota_1 \eta - \nu \iota_2 \sigma}.
\end{align} 
All these can be absorbed in $\mu_R$, as we will explain below.

The remainder $\mu_{R,0}(\xi,\eta,\s)$ in \eqref{muprop5} 
can be written as a linear combination of terms of two types: one type of terms is of the form 

\begin{align}\label{muprop6}
I_{\eps_2 \eps_3} 
  := \int_{\R} a(x) \chi_+(x) \chi_{\epsilon_2}(x)\chi_{\epsilon_3}(x) 
  {\overline{K_{o} (x,\xi)}
   K_{o}(\epsilon_2 x,\eta) K_{o}(\epsilon_3 x, \sigma) } \,dx,
\end{align}

where $\epsilon_2, \epsilon_3 \in \{+,-\}$ with $\eps_2\cdot\eps_3 = -1$,
while the other type of terms have the form 
\begin{align}\label{muprop7}
II_{A,B,C} 
  := \int_{\R_+} a(x) (\chi_+(x))^3 \overline{K^{A}_o(x,\xi)} 
  K^{B}_{o}(x,\eta) K^{C}_{o}(x,\s) \,dx,
\end{align}
where $A,B,C \in \{ S,R \}$ with at least one of them equal to $R$. 

For the terms of the type \eqref{muprop6} we notice that $a\chi_{+}\chi_{\epsilon_2}\chi_{\epsilon_3}$
is compactly supported so that the property \eqref{estmu1}-\eqref{estmu2} with $a=b=c=0$ 
and with $N=0$ follows from \eqref{mupropK1}, 
while for the terms in \eqref{muprop7}, it follows using also \eqref{mupropK2}.
In order to obtain \eqref{estmu2} for non-zero $a,b,c$ and for general $N$, 
it suffices to focus on one specific instance, such as the term
\begin{align}\label{muprop10}
\int a(x) (\chi_+(x))^3 K_o^{S}(x,\xi) K_o^{R}(x,\eta) K_o^{S}(x,\s) \,dx, 
\end{align}
since all other cases can be treated similarly.
From \eqref{Kosplit} we write
\begin{align}\label{Kreg1}
\begin{split}
& \sqrt{2\pi} \, K_o^S(x,\xi) = \frac{\xi}{\jxi} g(\xi) e^{ix\xi} + 2i \sin(x\xi), 
  \qquad g(\xi) :=  \frac{\jxi}{\xi} \big( T(\xi) - 1 - R(\xi) \big),
\\
& \sqrt{2\pi} \, K_o^R(x,\xi) = \frac{\xi}{\jxi} \, \Big[ g(\xi) (m_+(x,\xi)-1)e^{ix\xi} 
  + \int_{-1}^1 \partial_\rho \big[ e^{ix\rho} \langle \rho \rangle (m_+(x,\rho)-1) \big] (x,s\xi) ds \Big].
\end{split}
\end{align}
Observe that the function $g$ is smooth in view of \eqref{estimatesTRodd}.
Plugging the formulas \eqref{Kreg1} into \eqref{muprop10} gives various contributions;
we single out
\begin{align}\label{muprop11}
J(\xi,\eta,\s) := \frac{\xi \cdot \eta \cdot \s}{\jxi \jeta\jsig} 
  \bar{g(\xi)}g(\eta)g(\s) \int_\R (\chi_+(x))^3 e^{ix(-\xi+\eta+\s)} (m_+(x,\eta)-1) \,dx, 
\end{align}
which is representative of all the them; 
indeed, all the other terms from \eqref{muprop10} would either involve the $\sin(x\xi)$ term
(which we bound by $|\xi|\jx/\jxi$) or the integral in \eqref{Kreg1},
but these can be handled similarly using the estimates for 
$m_+$ in \eqref{estimatesm}.
To prove \eqref{estmu1}-\eqref{estmu2} with $a,b,c=0$ we write
$e^{ix(-\xi+\eta+\sigma)} = (i(-\xi+\eta+\sigma))^{-N} \partial_x^N e^{ix(-\xi+\eta+\sigma)}$ 
integrate by parts repeatedly in \eqref{muprop11}, and use \eqref{estimatesm} to get
\begin{align*}
\Big| \int_\R (\chi_+(x))^3 e^{ix(-\xi+\eta+\s)} (m_+(x,\eta)-1) \,dx \Big|
  \lesssim  \langle -\xi+\eta+\sigma \rangle^{-N}.
\end{align*}
To prove \eqref{estmu1}-\eqref{estmu2} for general $a,b,c$ it suffices to apply derivatives
to the integral in \eqref{muprop11}, use the smoothness of $g$, 
and then integrate by parts in $x$ as above, using once again \eqref{estimatesm}.

Finally, observe that the above argument works identically for terms as in \eqref{muprop6};
the same holds true if we replace $a\chi_{+}\chi_{\epsilon_2}\chi_{\epsilon_3}$
with $\psi$ or $(a - \ell_{+\infty})\chi_+^3$, which takes care of the terms \eqref{mupropR1} and \eqref{mupropR2}.

Eventually, we look at \eqref{mupropR3}. 
Note that, from the definition of $\varphi^\ast$ in \eqref{propmustar},
we have $|p| \gtrsim 1/R(\eta,\s)$ on its support.
We observe that
\begin{align*}
\begin{split}
\mu_{R,3}(0,\eta,\sigma) 
  & = \sum_{\lambda, \mu, \nu} \bar{a_{\lambda} (0)} a_{\mu}(\eta) a_{\nu} (\s)
  \ell_{+\infty} i \big[1-\varphi^{\ast}(p_0,\eta,\s)\big] \, \frac{\what{\phi}(p_0)}{p_0} \, \Big|_{
  p_0 = - \mu \iota_1 \eta - \nu \iota_2 \sigma}
  = 0,
\end{split}
\end{align*}
since $\bar{a_{\lambda} (0)}=\lambda$.
Similarly, the expression vanishes at $\eta=0$ or $\s=0$ and in particular
\begin{align*}
\frac{\jxi\jeta\jsig}{\xi \cdot \eta \cdot \s} \mu_{R,3}(\xi,\eta,\sigma) 
\end{align*}
is bounded. This shows \eqref{estmu1}-\eqref{estmu2} for $a=b=c=0$,
using also that $\phi \in \mathcal{S}$.
Finally, we notice that when differentiating \eqref{mupropR3} 
the worst terms are those where the cutoff $\varphi^\ast$ is hit, and since we can bound
\begin{align*}
\big| \partial_\xi^a\partial_\eta^b\partial_\s^c \, \varphi_{> - D}\big(pR(\eta,\s)\big) \big|
  \lesssim |R(\eta,\s)|^{a+b+c}
\end{align*}
the claimed bounds follow.
\end{proof}

\smallskip
\subsubsection*{The even case}
A result similar to Proposition \ref{propmuodd} holds in the case of even functions 
under the assumption that the operator $H$ has an odd resonance: 

\begin{proposition}[The even case]\label{propmueven} 
If $f, g \in \mathcal{S}$ are even functions, $a=a(x)$ is even and satisfies \eqref{ab},
the same statement as that of Proposition \ref{propmuodd} holds true, 
up to modifying the coefficients $a_{\pm}(\xi)$ in \eqref{propmuScoeff} as follows:
\begin{align}\label{propmuevencoeff}
a_+(\xi) = T(\xi) + R(\xi), \qquad \mbox{and} \qquad a_-(\xi) = 1. 
\end{align}

\end{proposition}

\begin{proof}
The proof can be obtained very similarly to the proof of Proposition \ref{propmuodd},
starting from the distorted Fourier transform formulas \eqref{invFeve}-\eqref{invFeve'}.
From \eqref{Keven} we can write formulas analogous to \eqref{Kosplit} and \eqref{Kreg1}, namely
\begin{align}\label{Kesplit}
\begin{split}
& K_e(x,\xi) =  K_e^{S} (x,\xi) + K_e^{R} (x,\xi),
\\
& \sqrt{2\pi}K_e^{S}(x,\xi) = T(\xi) e^{ix\xi} + e^{-ix\xi}  + R(\xi) e^{ix\xi} 
  = \sum_{\l\in\{+,-\}} a_\l(\xi)e^{i \l x\xi},
\\
& \sqrt{2\pi}K_e^{R}(x,\xi) = (T(\xi)+R(\xi))(m_+(x,\xi)-1)e^{ix\xi} + (m_+(x,-\xi)-1)e^{-ix\xi},
\end{split}
\end{align}
where $a_\l$ are the coefficients given in \eqref{propmuevencoeff}.
Then using \eqref{TRKeve}, 
\eqref{estimatesTR}, and \eqref{estimatesm}, we have the analogue of \eqref{mupropK1}-\eqref{mupropK2},
that is,
\begin{align}\label{mupropK1eve}
\begin{split}
& |K_e(x,\xi)|  
+ |K_e^{S}(x,\xi)|   \lesssim \min(|\xi|,1) + \min (|x\xi|,1),
\\
& |K_e^{R}(x,\xi)| \lesssim 
  \jx^{-N+1} \min(|\xi|,1),
\end{split}
\end{align}
on the support of $\chi_+(x)$,
and we can write, similarly to \eqref{Kreg1}, 
\begin{align}\label{Kreg1eve}
\begin{split}
& \sqrt{2\pi} \, K_e^S(x,\xi) = \frac{\xi}{\jxi} g(\xi) e^{ix\xi}  - 2i \sin(x\xi) , 
  \qquad g(\xi) :=  \frac{\jxi}{\xi} \big( T(\xi)+ R(\xi) + 1 \big),
\\
& \sqrt{2\pi} \, K_e^R(x,\xi) = \frac{\xi}{\jxi} \, \Big[ g(\xi) (m_+(x,\xi)-1)e^{ix\xi} 
  + \int_{-1}^1 \partial_\rho \big[ e^{ix\rho} \langle \rho \rangle (m_+(x,\rho)-1) \big] (x,s\xi) ds \Big].
\end{split}
\end{align}
Using \eqref{Kesplit}-\eqref{Kreg1eve} 
and noticing that $g$ is smooth in view of \eqref{estimatesTReven}, the proof can then proceed as before.
\end{proof}

\smallskip
\subsubsection*{The generic case} 
In this case we can directly borrow Proposition 4.1. from \cite{KGV1d}, 
which gives the following statement:\footnote{Compared to \cite{KGV1d}, we changed 
slightly the definitions of the singular and regular parts,
in accordance with those in Proposition \ref{propmuodd}
and to simplify the notation in the forthcoming sections.}

\begin{proposition}\label{propmug}
Let $f,g \in \mathcal{S}$ be arbitrary functions, 
with $a=a(x)$ satisfying \eqref{ab} (but no parity assumptions) and assume that the potential $V$ is generic. 
Then there exists a distribution $\mu_{\iota_1\iota_2}$, $\iota_1,\iota_2 \in \{+,-\}$, 
such that
\begin{align}
\label{propmug2}
\widetilde{\mathcal{F}} \big(a \, f_{\iota_1} g_{\iota_2} \big)(\xi) = \iint_{\R^2} 
  (\wt{f})_{\iota_1} (\eta) (\wt{g})_{\iota_2} (\sigma)\,\mu_{\iota_1 \iota_2}(\xi,\eta,\sigma) \, d\eta \,d\sigma,
\end{align}
and such that $\mu^o_{\iota_1\iota_2}$ can be split into a singular and a regular part as follows:
\begin{align*}
(2\pi) \mu_{\iota_1\iota_2} = \mu^{S}_{\iota_1\iota_2} + \mu^{R}_{\iota_1\iota_2},
\end{align*}
where:

\begin{itemize}

\item The singular part $\mu^{S}_{\iota_1\iota_2}$ is given by
\begin{align}\label{propmugS}
\begin{split}
\mu^{S}_{\iota_1 \iota_2}(\xi,\eta,\sigma) & = \sum_{\lambda, \mu, \nu, \epsilon} 
\bar{a_{\lambda}^\epsilon (\xi)} (a^\epsilon_{\mu}(\eta))_{\iota_1} (a^\epsilon_{\nu} (\s))_{\iota_2}
  \\ & \times \ell_{\epsilon \infty} \left[ \sqrt{\frac{\pi}{2} }\delta(p) - \,i \,\epsilon \, \varphi^{\ast}(p,\eta,\s)\, \pv \frac{\what{\phi}(p)}{p} \right]_{p = \lambda \xi - \mu \iota_1 \eta - \nu \iota_2 \sigma}
\end{split}
\end{align}
where the coefficients are given as
\begin{align}\label{propmugcoeff}
\left\{ \begin{array}{l}
\mathbf{a}^-_+(\xi) = \mathbf{1}_+(\xi) + \mathbf{1}_-(\xi) T(-\xi) \\
\mathbf{a}^-_-(\xi) = \mathbf{1}_+(\xi) R_-(\xi)
\end{array} \right. 
\qquad
\left\{ \begin{array}{l}
\mathbf{a}^+_+(\xi) = T(\xi) \mathbf{1}_+(\xi) + \mathbf{1}_-(\xi) \\
\mathbf{a}^+_-(\xi) = \mathbf{1}_-(\xi) R_+(-\xi),
\end{array} \right. 
\end{align}
$\phi \in \mathcal{S}$ is even with integral one, and 
\begin{align}\label{propmugstar}
\varphi^\ast(p,\eta,\sigma) = \varphi_{\leq -D}\big(p R(\eta,\sigma)\big),
\quad R(\eta,\sigma) = \frac{\langle \eta \rangle \langle \sigma \rangle}{\langle \eta \rangle + \langle \sigma \rangle}.
\end{align}
Here $\mathbf{1}_\pm$ is the characteristic function of $\{ \pm \xi \geq 0 \}$.

\medskip
\item The regular part $\mu^{R}_{\iota_1\iota_2}$ 
can be written as a linear combination of terms of the form
\begin{equation}
\label{estmu1g}
\mathbf{1}_{\eps_1}(\xi) \mathbf{1}_{\eps_2}(\eta) \mathbf{1}_{\eps_3} (\sigma)
  \frac{\xi \cdot \eta  \cdot \s}{\jxi\jeta\jsig} 
  \mathfrak{q}_{\substack{\iota_1 \iota_2 \\ \eps_1\eps_2\eps_3}}(\xi,\eta,\sigma)
\end{equation}
where $\eps_1,\eps_2,\eps_3 \in \{+,-\}$, and for all $|a|+|b|+|c| \leq N$
\begin{equation}
\label{estmu2g}
| \partial_\xi^a \partial_\eta^b \partial_\sigma^c 
  \mathfrak{q}_{\substack{\iota_1 \iota_2 \\ \eps_1\eps_2\eps_3}}(\xi,\eta,\sigma) | 
  \lesssim \sup_{\mu,\nu} \frac{1}{\langle \xi +\mu \eta +\nu \sigma \rangle^N} |R(\eta,\s)|^{1+a+b+c}.
\end{equation}

\end{itemize}

\end{proposition}

\medskip
\begin{rem}[Reduction to the case of odd symmetry]
From now on we will work under the odd symmetry assumption for \eqref{maineq}, that is, we assume the initial data 
is odd, the coefficient $a=a(x)$ in \eqref{maineq} is odd 
(so that the solution stays odd for all times), and the potential is even with (possibly)
an even zero energy resonance.

It is apparent from Propositions \ref{propmuodd} and \ref{propmueven} that the case of even data/solution and
an odd resonance can be dealt with in exactly the same way as the odd case.

As for the generic case, a little more care would be needed to deal with the non-smoothness of the coefficients
\eqref{propmugcoeff} at zero; in particular, after applying the normal form transformation \eqref{profile1}, 
one needs to show that the singularities in the cubic symbols in \eqref{CubicS}-\eqref{cubicSm}
only appear in the arguments of the inputs $(\eta,\s,\theta)$.
This is a technical point that requires some careful algebra, 
but since it was addressed already in Section 6 of \cite{KGV1d} we can skip it here.
A part from this, Proposition \ref{propmug} shows that the generic case can be handled like the odd case as well.
\end{rem}


\smallskip
\section{The main nonlinear decomposition and bootstrap}\label{MND}
In Subsections \ref{ssecprof}-\ref{SsecRenoeq} we summarize several manipulations 
which lead to a renormalized form of the equation (see \eqref{Renodtf}) over which the main estimates are performed. 
Some details are omitted, for which we refer the reader to Section 5 of \cite{KGV1d}.
In Subsection \ref{BABAPB} we state our main bootstrap proposition which will imply 
the global bounds of Theorem \ref{mainthm}.


\smallskip
\subsection{The equation on the profile}\label{ssecprof}
Consider $u = u(t,x)$ a solution of the quadratic Klein-Gordon equation
\begin{align}
\label{KGu}
\begin{split}
& \partial_{t}^2 u + (-\partial_x^2 + V + 1) u = a(x)u^2,  \qquad (u,u_t)(t=0) = (u_0,u_1).
\end{split}
\end{align}
Note that we are disregarding the cubic terms from \eqref{maineq} since they are lower order;
moreover, cubic terms that are more complicated than $u^3$,
will appear after normal forms, and will be treated in detail in what follows.

In order to make the equation first order in time, we first define 
\begin{align}\label{v}
v= \big(\partial_t - i B\big)u, \qquad B = \sqrt{-\partial_x^2 +V + 1},
\end{align}
which solves
$$
\big(\partial_t + i B \big)v = a(x)u^2 \quad 
  \mbox{or} \quad \big(\partial_{t} + i \langle \xi \rangle \big)\wt{v} = \wt{\mathcal{F}}(a(x)u^2)
$$
Next, we filter by the linear evolution to obtain the profile 
\begin{equation}\label{profile0}
g(t,\cdot) = e^{itB}v(t,\cdot) 
\end{equation}
which solves
$$ 
\partial_t \wt{g}(t,\xi) = e^{it\langle \xi \rangle} \wt{\mathcal{F}}(a(x)u^2).
$$
Using the quadratic spectral distribution described in Proposition \ref{propmuodd} 
we write this explicitly as
\begin{align}\label{dtg}
\begin{split}
\partial_t \wt{g}(t,\xi) & = - \sum_{\iota_1,\iota_2}  \iota_1\iota_2
\iint e^{it \Phi_{\iota_1\iota_2}(\xi,\eta,\sigma)} \wt{g}_{\iota_1}(t,\eta) \wt{g}_{\iota_2}(t,\sigma) 
\, \frac{1}{4\jeta \jsig} \mu_{\iota_1\iota_2}^o(\xi,\eta,\sigma) \, d\eta \, d\sigma.
\\
& \Phi_{\iota_1 \iota_2}(\xi,\eta,\sigma) = \jxi - \iota_1\jeta - \iota_2 \jsig.
\end{split}
\end{align}
For convenience we are omitting  the limits of integration $\eta,\s>0$
and, from now on, we will also omit the apex $o$ from $\mu_{\iota_1\iota_2}^o$ and similar expression
such as $\mu^{o,S}_{\iota_1\iota_2}$.

\subsection{Normal form transformation and renormalized profile}\label{secNF}
Some simple calculations show that $\Phi_{\iota_1 \iota_2}$ 
does not vanish on the support of the distribution $\mu^S_{\iota_1 \iota_2}$ defined in \eqref{propmuS} and \eqref{propmustar}. 
In other words, the corresponding interaction is not resonant, 
and we can define the natural normal form transformation $T_{\iota_1 \iota_2}(g,g)$ by
\begin{align}\label{defTiota}
\begin{split}
\wt{\mathcal{F}} T_{\iota_1, \iota_2}(g,g)(t) &:= \iint e^{it \Phi_{\iota_1 \iota_2}(\xi,\eta,\sigma)}
  \widetilde{g}(t,\eta) \widetilde{g}(t,\sigma) \mathfrak{m}_{\iota_1 \iota_2}(\xi,\eta,\sigma)\,d\eta \,d\sigma 
\\
& \mathfrak{m}_{\iota_1 \iota_2}(\xi,\eta,\s) 
  = -\frac{\iota_1\iota_2}{\jeta\jsig} \frac{\mu^{S}_{ \iota_1 \iota_2 }(\xi,\eta,\sigma)}{i \Phi_{\iota_1 \iota_2}(\xi,\eta,\sigma)}.
\end{split}
\end{align}
The full normal form transformation is given by $\sum_{\iota_1 \iota_2\in\{+,-\}} T_{\iota_1 \iota_2}$,
and accordingly we define a {\it re-normalized profile}
\begin{align}\label{profile1}
f := g - T(g,g), \qquad T = \sum_{\iota_1,\iota_2\in\{+,-\}} T_{\iota_1 \iota_2}.
\end{align}
It is not hard to check that, under our assumptions, $f$ is odd and thus $\wt{f}(t,0) = 0$ as well.

Let us write $\partial_t \wt{g} = \mathcal{Q}^R(g,g) + \mathcal{Q}^S(g,g)$ where $\mathcal{Q}^\ast$ denotes the bilinear expression of the same form of \eqref{dtg}
with $\mu^o$ replaced by $(2\pi)^{-1}\mu^\ast$, for $\ast = S$ or $R$.
Then, from \eqref{defTiota}-\eqref{profile1} we see that
$\partial_t \wt{f} = \mathcal{Q}^R(g,g) + \wt{T}\big(\partial_tg,g) + \wt{T}\big(g,\partial_tg)$.
This last equation can be rewritten as
\begin{align}\label{Renodtf0}
\begin{split}
\partial_t \wt{f} = \mathcal{Q}^R(g,g) + \mathcal{C}^{S}(g,g,g) + \mathcal{C}^R(g,g,g)
\end{split}
\end{align}
where the terms on the right-hand side are given below:

\medskip
\noindent
\setlength{\leftmargini}{1.5em}
\begin{itemize}

\item \textit{The regular quadratic term} is  
 \begin{align}\label{QR}
\begin{split}
& \mathcal{Q}^R(a,b) = \sum_{\iota_1,\iota_2\in\{+,-\}} \mathcal{Q}_{\iota_1\iota_2}^R(a,b),
\\
& \mathcal{Q}_{\iota_1\iota_2}^R[a,b](t,\xi)
  = \iint e^{it \Phi_{\iota_1\iota_2}(\xi,\eta,\sigma)} \, \mathfrak{q}(\xi,\eta,\sigma) 
  \, \wt{a}_{\iota_1}(t,\eta) \wt{b}_{\iota_2}(t,\sigma) \, d\eta \, d\sigma,
\end{split}
\end{align}
with $\mathfrak{q}$ satisfying, see \eqref{estmu1}-\eqref{estmu2} and \eqref{propmustar},
\begin{align}\label{qsymbol} 
\mathfrak{q}(\xi,\eta,\s) = \mathfrak{q}'(\xi,\eta,\s) \cdot \frac{\eta}{\jeta} \frac{\s}{\jsig} 
  \cdot \frac{1}{\jeta+\jsig}
\end{align}
where $\mathfrak{q}'$ is smooth with
\begin{align}\label{q'symbol}
|\partial_\xi^a \partial_\eta^b \partial_\s^c\mathfrak{q'}(\xi,\eta,\s)| \lesssim \langle \xi-\eta-\sigma \rangle^{-N}
  \min(\jeta, \jsig)^{a+b+c}.
\end{align}
Note that the actual bound above should have the factor $(\inf_{\mu,\nu}  \langle \xi- \mu \eta- \nu \sigma \rangle)^{-N}$
but we disregard the signs $\mu,\nu$ since they will play no relevant role in our estimates.


\smallskip
\item \textit{The singular cubic term} is
\begin{align}
\label{CubicS}
\begin{split}
& \mathcal{C}^{S}(a,b,c) = \sum_{\kappa_1, \kappa_2, \kappa_3 \in\{+,-\}} 
  \mathcal{C}^{S}_{\kappa_1 \kappa_2 \kappa_3}(a,b,c) 
\\
& \mathcal{C}^{S}_{\kappa_1 \kappa_2 \kappa_3}(a,b,c)(t,\xi)
= \iiint e^{it \Phi_{\kappa_1 \kappa_2 \kappa_3}(\xi,\eta,\sigma,\theta)} 
  \mathfrak{c}^S_{\kappa_1 \kappa_2 \kappa_3}(\xi,\eta,\sigma,\theta) 
  \,  \wt{a}_{\k_1}(t,\eta) \wt{b}_{\k_2}(t,\sigma) \wt{c}_{\k_3}(t,\theta) \, d\eta \, d\sigma \,d\theta,
\\
& \Phi_{\kappa_1 \kappa_2 \kappa_3}(\xi,\eta,\sigma,\theta) := \jxi - \kappa_1\jeta - \kappa_2\jsig 
  - \kappa_3\langle \theta \rangle,
\end{split}
\end{align}
where the symbol $\mathfrak{c}^S_{\kappa_1 \kappa_2 \kappa_3}$ can be written as a sum of terms of the type
\begin{align}\label{cubicSm}
\mathfrak{s}(\xi,\eta,\sigma,\theta) \delta(p) 
  \qquad \mbox{and} \qquad \mathfrak{m}(\xi,\eta,\sigma,\theta) \frac{\widehat{\phi}(p)}{p}, 
\qquad
p := \lambda \xi - \mu \eta - \nu \sigma - \rho \theta,
\end{align}
where we write $\mathfrak{s}$ and $\mathfrak{m}$ for symbols which are globally Lipschitz,
and smooth as long as $\xi,\eta,\sigma,\theta$ do not vanish. 
Furthermore, they satisfy  
\begin{equation}\label{cubicSm'}
\begin{split}
& \left|\mathfrak{m}(\xi,\eta,\sigma,\theta) \right| + \left| \mathfrak{s}(\xi,\eta,\sigma,\theta) \right| 
  \lesssim \frac{1}{\langle \eta \rangle \langle \sigma \rangle \langle \theta \rangle} ,  
\\
& \text{$\mathfrak{m}(\xi,\eta,\sigma,\theta) = 0 $ \,\, when \, $\eta \cdot \sigma \cdot \theta = 0$.}  
\end{split}
\end{equation}
Precise bounds on derivatives of $\mathfrak{m}$ and $\mathfrak{s}$ are slightly more complicated to state.
However, all the contributions coming from differentiating these symbols (as in Section \ref{secCS} for example)
are lower order; in fact, the worst loss that may occur when differentiating them
is essentially a factor of the form $\max(\jxi,\jeta,\jsig,\langle\theta\rangle)$ which is easy to handle.
We refer to Subsection 5.5 in \cite{KGV1d} where exact formulas can be found.

\smallskip
\item \textit{The regular cubic term} is given by
\begin{align}
\label{CubicR}
\begin{split}
& \mathcal{C}^{R}(a,b,c) = \sum_{\kappa_1 \kappa_2 \kappa_3} \mathcal{C}^{R}_{\kappa_1 \kappa_2 \kappa_3}(a,b,c)
\\
& \mathcal{C}^{R}_{\kappa_1 \kappa_2 \kappa_3}(a,b,c)(t,\xi)
= \iiint e^{it \Phi_{\kappa_1 \kappa_2 \kappa_3}(\xi,\eta,\sigma,\theta)}
  \mathfrak{c}^R_{\kappa_1 \kappa_2 \kappa_3}(\xi,\eta,\sigma,\theta) 
  \, \wt{a}_{\iota_1}(t,\eta) \wt{b}_{\iota_2}(t,\sigma)  
  \wt{c}_{\iota_3}(t,\theta) \, d\eta \, d\sigma \,d\theta,
\end{split}
\end{align}
where the symbol $\mathfrak{c}^R_{\kappa_1 \kappa_2 \kappa_3}$ can be written as a sum of symbols 
$\mathfrak{m}$ satisfying
\begin{align}\label{CubicRm}
|\mathfrak{m}(\xi,\eta,\sigma,\theta)| \lesssim \frac{1}{\langle \eta \rangle \langle \sigma \rangle \langle \theta \rangle}.
\end{align}
This regular term is  easier to treat than the two previous ones. 
In particular, it satisfies better trilinear bounds than $\mathcal{C}^S$ does. 
Therefore, we will only briefly mention how to treat it in the rest of the proof.


\end{itemize}

\smallskip
We now reproduce Lemmas 6.11 and 6.13 from \cite{KGV1d} for later reference. 
The statements of these lemmas involve the wave operator associated to $H$, which is denoted
\begin{align}\label{Wopdef}
\mathcal{W} = \widetilde{\mathcal{F}}^{-1} \widehat{\mathcal{F}},
\qquad \mathcal{W}^\ast =  \widehat{\mathcal{F}}^{-1}\widetilde{\mathcal{F}}.
\end{align}

\begin{lem} \label{lemQR}
For any $p_1,p_2\in[2,\infty)$ such that $\frac{1}{p_1}+\frac{1}{p_2} < \frac{1}{2}$,
and $\iota_1,\iota_2 \in \{+,-\}$,
\begin{align}\label{lemQRpq}
\begin{split}
{\big\| 
  \mathcal{Q}^R_{\iota_1\iota_2}(f_1,f_2)(t,\xi) \big\|}_{L^2_\xi} 
  \lesssim \min\big( & {\| \langle \partial_x \rangle^{-1+} 
  e^{-i\iota_1t\langle \partial_x \rangle} \mathcal{W}^*f_1 \|}_{L^{p_1}} 
  {\| 
  e^{-i\iota_2t\langle \partial_x \rangle} \mathcal{W}^*f_2 \|}_{L^{p_2}},
  \\ & {\| 
  e^{-i\iota_1t\langle \partial_x \rangle} \mathcal{W}^*f_1 \|}_{L^{p_1}} 
  {\| \langle \partial_x \rangle^{-1+} 
  e^{-i\iota_2t\langle \partial_x \rangle} \mathcal{W}^*f_2 \|}_{L^{p_2}} \big).
\end{split}
\end{align}
\end{lem}

\begin{lem} \label{lemCS}
For any $\kappa_1, \kappa_2, \kappa_3\in \{+,-\}$, for all $p,p_1,p_2,p_3 \in (1,\infty)$ 
with $\frac{1}{p_1} + \frac{1}{p_2} + \frac{1}{p_3} = \frac{1}{p}$, 
\begin{align}\label{lemCS1}
\begin{split}
& {\|e^{-it\langle \partial_x \rangle}\widehat{\mathcal{F}}^{-1}\mathcal{C}^{S}_{\kappa_1 \kappa_2 \kappa_3}(a,b,c)\|}_{L^p}
 \\
 & \qquad \lesssim
 {\| \langle \partial_x \rangle^{-1+ } e^{-\kappa_1it \langle \partial_x \rangle} \mathcal{W}^* a \|}_{L^{p_1}}
 {\| \langle \partial_x \rangle^{-1+} e^{-\kappa_2it \langle \partial_x \rangle} \mathcal{W}^* b \|}_{L^{p_2}}
 {\| \langle \partial_x \rangle^{-1+} e^{-\kappa_3it\langle \partial_x \rangle} \mathcal{W}^* c \|}_{L^{p_3}}.
\end{split}
\end{align}
Furthermore, if $p_1 = p_2 = \infty$, and $a,b$ are functions that satisfy the
bounds as in the assumptions \eqref{bootstrap0} of Proposition \ref{propbootstrap} below,
then, for all $t\in[0,T]$ and 
$p\in(1,\infty)$, we have
\begin{align}\label{lemCSend}
& {\big\|e^{-it\langle \partial_x \rangle}
  \widehat{\mathcal{F}}^{-1}\mathcal{C}^{S}_{\kappa_1 \kappa_2 \kappa_3}[a , b, c] \big\|}_{L^p} 
  \lesssim \frac{\e_1^2}{t} 
  {\big\| \langle \partial_x \rangle^{-1+} e^{-it\langle \partial_x \rangle} \mathcal{W}^* c \big\|}_{L^{p}}.
\end{align}
\end{lem}

The estimates of Lemma \ref{lemCS} above hold true also for the $\delta$ and $\pv$
components of $\mathcal{C}^S$ separately, see \eqref{CubicS}-\eqref{cubicSm}.
Moreover, one can also add derivatives to the estimate \eqref{lemCS1} with a natural statement
consistent with product estimates in Sobolev spaces;
in particular we can replace the $L^p$ norm on the left-hand side of \eqref{lemCS1} by the $W^{0+,p}$ norm.

\smallskip
\subsection{The equation for the renormalized profile}\label{SsecRenoeq} 

Using the identity $g=f+T(g,g)$ (see \eqref{profile1}), we can write the equation \eqref{Renodtf0} as follows
\begin{align}\label{Renodtf}
\partial_t \wt{f} = \mathcal{Q}^R(f,f) + \mathcal{C}(f,f,f) + \mathcal{R}(f,g),
  \qquad \mathcal{C} := \mathcal{C}^{S} + \mathcal{C}^R,
\end{align}
where the remainder term $\mathcal{R}$ is given by
\begin{align}
\label{R}
\mathcal{R}(f,g) & = \mathcal{R}_Q(f,g) + \mathcal{R}_C(f,g),
\\
\label{RQ}
\mathcal{R}_Q(f,g) & = \mathcal{Q}^R(f,T(g,g)) + \mathcal{Q}^R(T(g,g),g),
\\
\label{RC}
\mathcal{R}_C(f,g) & = \mathcal{C}(T(g,g),g,g) + \mathcal{C}(f,T(g,g),g) + \mathcal{C}(f,f,T(g,g)).
\end{align}


All the terms above are simpler to estimate than the other quadratic or cubic terms in \eqref{Renodtf},
and can be treated similarly or with more straightforward arguments;
we will explain on how to handle them in Section \ref{secred}
and will mostly concentrate on $\mathcal{Q}^R(f,f)$ and  $\mathcal{C}^S(f,f,f)$ in what follows.



\medskip
\subsection{Bootstrap and basic a priori bounds}\label{BABAPB}

Fix two constants $\alpha$ and $p_0$ such that $0<p_0 \ll \alpha \ll 1$,
and recall the smallness assumption in Theorem \ref{mainthm}.
Our main bootstrap proposition, that will immediately imply the main theorem, is the following:

\begin{prop}\label{propbootstrap}
Assuming that
\begin{align}\label{bootstrap0}
\begin{split}
& \sup_{t \in [0,T]} \left[ \jt^{-p_0} {\| \jxi^4 \wt{f}(t) \|}_{L^2} 
  + \jt^{-\alpha} {\| \jxi \partial_\xi \wt{f}(t) \|}_{L^2} 
  + {\| \jxi^{3/2} \wt{f}(t) \|}_{L^\infty} \right] < 2 \varepsilon_1,
  \\
& \sup_{t \in [0,T]} \left[ \jt^{-p_0} {\| \jxi^4 \wt{g}(t) \|}_{L^2} + 
  \jt^{1/2} \| e^{-it\langle \partial_x \rangle} \mathbf{1}_{\pm}(D) \mathcal{W}^* g(t) \|_{L^\infty} \right]
  < 8\varepsilon_1,
\end{split}
\end{align}
with $\e_1 = C\e$ for some absolute constant $C>0$ sufficiently large.
Then, we have
\begin{equation}
\label{eqbootstrap}
\begin{split}
& \sup_{t \in [0,T]} \left[ \jt^{-p_0} {\| \jxi^4 \wt{f}(t) \|}_{L^2} 
  + \jt^{-\alpha} {\| \jxi \partial_\xi \wt{f}(t) \|}_{L^2} 
  + {\| \jxi^{3/2} \wt{f}(t) \|}_{L^\infty} \right] < \varepsilon_1,
  \\
  & \sup_{t \in [0,T]} \left[ \jt^{-p_0} {\| \jxi^4 \wt{g}(t) \|}_{L^2} 
  + \jt^{1/2} \| e^{-it\langle \partial_x \rangle} \mathbf{1}_{\pm}(D) \mathcal{W}^* g(t) \|_{L^\infty} \right]
  < 4\varepsilon_1.
\end{split}
\end{equation}
\end{prop}

\medskip
At the heart of the proof of Proposition \ref{propbootstrap} are the following weighted-type estimates for the quadratic and 
cubic terms in \eqref{Renodtf}, under the assumptions \eqref{bootstrap0}:
\begin{align}\label{bootQR}
{\Big\| \jxi \partial_\xi \int_0^t \mathcal{Q}^R(s,\xi) \, ds \Big\|}_{L^2} 
   + {\Big\| \jxi \partial_\xi \int_0^t \mathcal{C}^S(s,\xi) \, ds \Big\|}_{L^2} \lesssim \e_1^2 \jt^\alpha.
\end{align}
These estimates are proven in Sections \ref{sectionregular} and \ref{secCS}. 
The bounds on the other two norms - the Sobolev-type norm and the Fourier $L^\infty$ norm -
are discussed in Section \ref{secred}. 


\medskip
Here are some immediate consequences of the bootstrap assumptions:

\begin{lem} \label{ibis}
Under the assumptions of Proposition \ref{propbootstrap} we have

\begin{itemize}

\item[(i)] (Global decay of the solution)
For $v = (\partial_t - i B)u$ we have
\begin{align*}
\| v \|_{L^\infty} \lesssim \e_1 \langle t \rangle ^{-1/2}.
\end{align*}

\item[(ii)] (Decay for the derivative of the profile)

\begin{align}\label{dsfL^infty}
\|   e^{-itB}  \partial_t f \|_{L^\infty} \lesssim \e_1^2 \langle t \rangle^{-3/2 + 2\alpha},
\end{align}
and
\begin{align}\label{dsfL^2}
{\big\|  \partial_t f(t) \big\|}_{L^2} \lesssim \e_1^2   \jt^{-1+\alpha/2}  .
\end{align}


\end{itemize}

\end{lem}

\begin{proof}
Recall that $v = e^{-itB}g$. The first assertion follows from 
$g = f + T(g,g)$, the dispersive estimate \eqref{dec} applied to $e^{-itB}f$
with the a priori bounds \eqref{bootstrap0},
and a product estimate for the quadratic term $e^{-itB}T(g,g)$; we refer the 
reader to Proposition 7.1 in \cite{KGV1d} for the details.

Turning to (ii), recall that
\begin{align}\label{dsfpr1}
\begin{split}
e^{-isB} \partial_s f = \widetilde{\mathcal{F}}^{-1} 
  e^{-is\jxi} \mathcal{Q}^R(f,f) + 
  \widetilde{\mathcal{F}}^{-1} e^{-is\jxi}  \mathcal{C}^S(f,f,f) 
  \\
  + \, \widetilde{\mathcal{F}}^{-1} e^{-is\jxi}  \mathcal{C}^R(f,f,f) + \{ \mbox{better terms} \}.
\end{split}
\end{align}

We start with $\mathcal{Q}^R$, which can actually be seen to satisfy a better bound of the form $\js^{-2+2\alpha}$.
By Sobolev's embedding, it suffices to show that, for $s \geq 1$, we have
\begin{align}\label{dsfpr1'}
{\|  \jxi \mathcal{Q}^R(f,f)(\xi)  \|}_{L^2} \lesssim s^{-3/2}.
\end{align}

From \eqref{QR} and \eqref{qsymbol} we have 
\begin{align}\label{dsfpr2}
\begin{split}
I(s,\xi) :=  \jxi \mathcal{Q}^R(f,f)(\xi) 
  & = \iint e^{is \Phi(\xi,\eta,\sigma)} \, \xi \,\mathfrak{q}(\xi,\eta,\sigma) 
  \, \wt{f}(\eta) \wt{f}(\sigma) \, d\eta \, d\sigma
  \\
  & = 
  \iint_{\R^2_+} e^{is\Phi} \wt{f}(\eta) \wt{f}(\sigma) \, \frac{\eta}{\jeta} \frac{\s}{\jsig} \frac{\xi}{\jeta+\jsig}
  \mathfrak{q}'(\xi,\eta,\sigma) \,d\eta \, d\sigma,
\end{split}
\end{align}
where $\mathfrak{q}'$ satisfies \eqref{q'symbol}.
Note that we can use the factors of $\eta/\jeta$ and $\s/\jsig$ to integrate by parts in \eqref{dsfpr2}.
Such an argument is performed very similarly in Section \ref{sectionregular} to bound the term $L_m$
appearing in \eqref{reglow}; the only difference between \eqref{reglow} and the term above
is that \eqref{dsfpr2} does not have a low frequency cutoff $\chi^l_m$, 
and \eqref{reglow} has a factor of $s$ and it is integrated in time. 

Integrating by parts in \eqref{dsfpr2} will then give the main terms 
\begin{align}\label{dsfpr3}
\begin{split}
Q_1(s,\xi) & = \frac{1}{s^2}\iint e^{is \Phi(\xi,\eta,\sigma)}
  \partial_\eta \wt{f}(\eta) \partial_\s \wt{f}(\sigma) \frac{\jxi}{\jeta+\jsig} \mathfrak{q}'(\xi,\eta,\sigma) \,d\eta \, d\sigma,
  \\
Q_2(s,\xi) & = \frac{1}{s^2} \iint e^{is \Phi(\xi,\eta,\sigma)} 
  \partial_\eta \wt{f}(\eta) \wt{f}(\s)
  \frac{\jxi}{\jeta+\jsig} \partial_\s \mathfrak{q}'(\xi,\eta,\sigma) \,d\eta \, d\sigma.
\end{split}
\end{align}
Using \eqref{q'symbol}, gives
\begin{align*}
\nonumber
{\big\| Q_1(s) \big\|}_{L^2} & \lesssim s^{-2} 
  {\Big\| \iint \big| \partial_\eta \wt{f}(\eta) \big| \, 
  \big| \partial_\sigma \wt{f} (\sigma) \big|
  \, \frac{1}{\langle \xi-\eta-\sigma \rangle^{N-1}} \,d\eta \,d\sigma \Big\|}_{L^2_\xi}
  \\ 
  &   \lesssim s^{-2}  {\| \partial_\eta \wt{f} \|}_{L^2}    \cdot {\| \partial_\s \wt{f} \|}_{L^1} 
  \\
\nonumber  
& \lesssim s^{-2} \e_1^2 s^{2\alpha},
\end{align*}
which is more than sufficient. A similar estimate holds for $Q_2$.


The remaining cubic terms in \eqref{dsfpr1} can be treated using Lemma \ref{lemCS},
and the comment after its statement, to remedy the lack of the endpoint estimate with $p=\infty$
as follows:
choosing $1/p < s < \alpha$ 
and using successively Sobolev's embedding, Lemma \ref{lemCS}, 
the intertwining property and boundedness of the wave operator,
and, finally, Proposition \ref{propdecay0} together with interpolation 
to deduce decay of $e^{\pm itB} f$ in $L^{3p}$, we get
\begin{align*}
{\big\| \widetilde{\mathcal{F}}^{-1} e^{-it\jxi} \mathcal{C}^S(f,f,f) \big\|}_{L^\infty}
& \lesssim {\big\| \langle \partial_x \rangle^s \widetilde{\mathcal{F}}^{-1} 
  e^{-it\jxi} \mathcal{C}^S(f,f,f) \big\|}_{L^p} 
  \\
& \lesssim \| e^{\pm i t \langle \partial_x \rangle} \mathcal{W}^* f \|_{L^{3p}}^3 
  \lesssim \| \mathcal{W}^* e^{\pm itB} f \|_{L^{3p}}^3 
  \\
  & \lesssim {\| e^{\pm itB} f \|}_{L^{3p}}^3
  \lesssim \big( \varepsilon_1 t^{-1/2 + 1/(3p) + 2p_0/(3p)} \big)^3
  \lesssim \varepsilon_1^3 t^{-3/2 + 2\alpha}.
\end{align*}

Finally, the $L^2$ estimate \eqref{dsfL^2} is a consequence of the equation \eqref{dsfpr1} 
combined with the H\"older type bounds of Lemmas \ref{lemQR} and \ref{lemCS} and the a priori bounds \eqref{bootstrap0}.
\end{proof}

\begin{proof}[Proof of Theorem \ref{mainthm} from Proposition \ref{propbootstrap}]
To deduce the main theorem from the above proposition, 
the first step is to check that the bootstrap assumptions in \eqref{bootstrap0} are satisfied at $t=0$.
As far as $g$ is concerned, this can be done easily, since from \eqref{v}-\eqref{profile0} we have
$$g(t=0) = v (t=0) = u_1 - i B u_0.$$
Then, the assumptions on the second line of \eqref{bootstrap0} at $t=0$ 
hold true thanks to the smallness assumption in Theorem \eqref{mainthm},
and the boundedness of wave operators on Sobolev spaces, by choosing $\e_1 = C\e$ for $C$ large enough.
The verification of the assumptions for $f$ on the first line of \eqref{bootstrap0}
when $t=0$ is slightly more involved, since $f(t=0)$ is a nonlinear function of $g(t=0)$; 
for this we refer to \cite{KGV1d}, where this task is accomplished in estimates (7.1) to (7.5).

Assuming Proposition \ref{propbootstrap} and using a standard continuation argument 
together with Lemma \ref{ibis}, 
we obtain the main theorem.
\end{proof}


\medskip
\section{The main regular interaction}\label{sectionregular}
In this section we prove the weighted $L^2$ bound for the regular quadratic terms
\begin{align}\label{QRmainbound}
{\Big\| \jxi \partial_\xi \int_0^t \mathcal{Q}^R(s,\xi)  \, ds \Big\|}_{L^2} \lesssim 
  \e_1^2,
\end{align}
that is even stronger than \eqref{bootQR} and what is needed for \eqref{eqbootstrap}.
Applying $\partial_\xi$ to $\mathcal{Q}_R$ in \eqref{QR}, we see that the main term 
is the one where the derivative hits the phase, and therefore we can reduce matters to
estimating the expression
\begin{align}\label{reg1}
\begin{split}
& I_m(t,\xi) := \int_0^t \iint_{(\R_+)^2} s \xi \, 
  e^{is \Phi_{\iota_1 \iota_2}(\xi,\eta,\sigma)} \wt{f}_{\iota_1}(\eta) \wt{f}_{\iota_2} (\sigma)
  \mathfrak{q}_{\iota_1 \iota_2}(\xi,\eta,\sigma) \,d\eta \,d\sigma \, \tau_m(s) \,ds,
\\
& \Phi_{\iota_1 \iota_2}(\xi,\eta,\sigma) := \langle \xi \rangle -  
  \iota_1 \langle \eta \rangle - \iota_2 \langle \sigma \rangle, \qquad \iota_1,\iota_2 \in \{+,-\},  
\end{split}
\end{align}
and proving that, for any fixed $m$, we have
\begin{align}\label{secregmain}
{\| I_m(t) \|}_{L^2} \lesssim 2^{-\delta m} \e_1^2,
\end{align}
for some $\delta >0$ small enough.
For simplicity of notation we have omitted the dependence on the indexes $\iota_1,\iota_2$ in the definition of $I_m$,
and we will often disregard them in what follows since they play no major role;
the understanding is that the case $\iota_1=\iota_2=+$ is relatively harder then the rest, 
and it suffices to concentrate on this.

\subsubsection*{Decomposition of \eqref{reg1}}
Let us define the multipliers
\begin{align}\label{regcuts}
\begin{split}
& \chi^{l}_m(\xi,\eta,\sigma) := 1 - \varphi_{>J} (\xi) \varphi_{>J} (\eta) \varphi_{>J} (\sigma), 
\\
& \chi^{h}_m (\xi,\eta,\sigma) := 1 - \chi^{l}_m(\xi,\eta,\sigma) , \qquad   \qquad J := -(10\alpha) m.
\end{split}
\end{align}
On the support of $\chi^l$ we have $\min(|\xi|,|\eta|,|\sigma|) \lesssim 2^J$, 
while on the support of $\chi^h$ we have $\min(|\xi|,|\eta|,|\sigma|) \gtrsim 2^J$.
Using these, we split \eqref{reg1} as follows: 
\begin{align}\label{reg2}
I_m = L_m + H_m + R_m, 
\end{align}
where
\begin{align}\label{reglow}
L_m :=  \int_0^t \iint e^{is \Phi(\xi,\eta,\sigma)} s \xi\, 
   \widetilde{f}(\eta) \widetilde{f} (\sigma) \chi^{l}_m(\xi,\eta,\sigma) \mathfrak{q}(\xi,\eta,\sigma)
  \,d\eta \,d\sigma \, \tau_m(s) \,ds
\end{align}
is the low frequency term,
\begin{align}\label{regres}
H_m :=  \int_0^t \iint e^{is \Phi(\xi,\eta,\sigma)} s \xi\, 
   \widetilde{f}(\eta) \widetilde{f} (\sigma) \chi^{h}_m(\xi,\eta,\sigma) \mathfrak{q}(\xi,\eta,\sigma) 
   \varphi_{<2J}(\Phi)
  \,d\eta \,d\sigma \, \tau_m(s)\,ds
\end{align}
is an almost resonant term, and the remainder $R_m$ is
\begin{align}\label{regnonres}
& 
R_m := \int_0^t \iint e^{is \Phi(\xi,\eta,\sigma)} s \xi\, 
  \widetilde{f}(\eta) \widetilde{f}(\sigma) \chi^{h}_m(\xi,\eta,\sigma) \mathfrak{q}(\xi,\eta,\sigma) 
  \varphi_{\geq 2J}(\Phi)  \tau_m(s) 
  \,d\eta \,d\sigma \, \tau_m(s)\,ds. 
\end{align}



\smallskip
\subsubsection*{Estimate of \eqref{reglow}}
To estimate \eqref{reglow} we recall \eqref{qsymbol}-\eqref{q'symbol} 
and write
\begin{align}\label{regq'}
\begin{split}
& \mathfrak{q}(\xi,\eta,\s) = \mathfrak{q}'(\xi,\eta,\s) \cdot \frac{\eta}{\jeta} \cdot \frac{\s}{\jsig},
\\
& |\partial_\eta^b \partial_\s^c\mathfrak{q'}(\xi,\eta,\s)| \lesssim \langle \xi-\eta-\sigma \rangle^{-N}
  \min(\jeta, \jsig)^{b+c} \frac{1}{\jeta+\jsig}.
\end{split}
\end{align}
Note that we have disregarded the signs $\mu,\nu$ in \eqref{estmu2} since these play no relevant role.
We then integrate by parts in $\sigma$ through the identity 
$e^{is \jsig} = \jsig(i s \sigma)^{-1} \partial_\sigma e^{is \jsig}$, 
and similarly in $\eta$, noting that the boundary terms vanish because $\wt{f}(0)=0$.
Out of the terms arising from this operation, 
we single out the following ones, since the remaining terms are either similar or easier:
\begin{subequations}
\begin{align}
\label{macreuse1} 
L_{m,1} & := \int_0^t \iint_{(\R_+)^2} e^{is \Phi} \frac{1}{s} \xi
  \, \chi^{l}_m(\xi,\eta,\sigma) 
  \partial_\eta \widetilde{f}(\eta) \partial_\sigma \widetilde{f} (\sigma) \mathfrak{q}'(\xi,\eta,\sigma) 
  \,d\eta \,d\sigma \, \tau_m(s) \,ds ,
\\
\label{macreuse2} 
L_{m,2} & := \int_0^t \iint_{(\R_+)^2} e^{is \Phi} \frac{1}{s} \xi \, 
  \chi^{l}_m(\xi,\eta,\sigma)  \, \partial_\eta \wt{f}(\eta) \wt{f} (\sigma) \, 
  \partial_\s \mathfrak{q}'(\xi,\eta,\sigma) \,d\eta \,d\sigma \, \tau_m(s) \,ds ,
\\
\label{macreuse3}  
L_{m,3} & := \int_0^t \iint_{(\R_+)^2} e^{is \Phi} \frac{1}{s} \xi \, 
  \partial_\sigma  \chi^{l}_m(\xi,\eta,\sigma)
  \partial_\eta \wt{f}(\eta) \wt{f} (\sigma) \mathfrak{q}'(\xi,\eta,\sigma) \,d\eta \,d\sigma \, \tau_m(s) \,ds .
\end{align}
\end{subequations}

To estimate \eqref{macreuse1} in $L^2$ we first use \eqref{regq'} to bound 
\begin{align*}
& \big| L_{m,1}(t,\xi) \big| \lesssim \sup_{s\approx 2^m} 
  \iint_{(\R_+)^2}  \chi^{l}_m(\xi,\eta,\sigma) 
  \big| \partial_\eta \wt{f}(\eta) \big| \, \big| \partial_\sigma \wt{f} (\sigma) \big|
  \, K(\xi,\eta,\s) \,d\eta \,d\sigma,
  \\
  & K(\xi,\eta,\s) := \frac{|\xi|}{\jeta + \jsig} \frac{1}{\langle\xi-\eta-\sigma \rangle^N}.
\end{align*}
Note that $K(\xi,\eta,\s) \lesssim {\langle\xi-\eta-\sigma \rangle}^{-N+1}$,
and recall that $\min\{|\xi|, |\eta|, |\s| \} \lesssim C2^J$. 
For the part $ |\s| \lesssim C2^J$, 
with Young's inequality followed by Cauchy-Schwarz we obtain
\begin{align}
\nonumber
{\big\| L_{m,1}(t,\cdot) \big\|}_{L^2} & \lesssim \sup_{s\approx 2^m} 
  {\Big\| \iint_{(\R_+)^2} \big| \partial_\eta \wt{f}(\eta) \big| \, 
  \varphi_{< J + D}(\s) \big| \partial_\sigma \wt{f} (\sigma) \big|
  \, \frac{1}{\langle \xi-\eta-\sigma \rangle^{N-1}} \,d\eta \,d\sigma \Big\|}_{L^2_\xi}
  \\ 
\label{regreuse}
& \lesssim \sup_{s\approx 2^m} {\| \partial_\eta \wt{f} \|}_{L^2}  \cdot {\| \varphi_{< J + D}(\s)\partial_\s \wt{f} \|}_{L^1} 
  \\
\nonumber  
& \lesssim \sup_{s\approx 2^m}  {\| \partial_\eta \wt{f} \|}_{L^2}  \cdot 2^{J/2} {\| \partial_\s \wt{f} \|}_{L^2} 
  \lesssim \e_1^2 2^{-3m\alpha}.
\end{align}
The case $|\eta| \lesssim C2^J$ can be dealt with in the same way due to symmetry. 
For the case $ |\xi| \lesssim C2^J$, 
we have $K(\xi,\eta,\s) 
\lesssim 2^J {\langle\xi-\eta-\sigma \rangle}^{-N} $. 
Hence by Young's inequality and Cauchy-Schwarz
\begin{align*}
{\big\| L_{m,1}(t,\cdot) \big\|}_{L^2} & \lesssim \sup_{s\approx 2^m} 
  {\Big\| \iint_{(\R_+)^2} \big| \partial_\eta \wt{f}(\eta) \big| \, 
    \big| \partial_\sigma \wt{f} (\sigma) \big|
  \, \frac{2^J}{\langle \xi-\eta-\sigma \rangle^{N}} \,d\eta \,d\sigma \Big\|}_{L^2_\xi}
  \\ 
& \lesssim 2^J \sup_{s\approx 2^m} {\| \partial_\eta \wt{f} \|}_{L^2}  \cdot {\|  \partial_\s \wt{f} \|}_{L^1} 
  \lesssim \e_1^2 2^{-3m\alpha}.
\end{align*}

A similar argument applies to \eqref{macreuse2} since $|\partial_\sigma q'| \lesssim \langle \xi- \eta - \sigma \rangle^{-N}$. 

For \eqref{macreuse3} we notice that $\partial_\sigma  \chi^{l}_m(\xi,\eta,\sigma) 
= -\varphi_{>J} (\xi) \varphi_{>J} (\eta) \varphi_{\sim J} (\sigma) 2^{-J}$,
and use again \eqref{regq'} and Young's inequality to get
\begin{align*}
{\| L_{m,3}(t) \|}_{L^2} & \lesssim \sup_{s\approx 2^m} 
  \iint_{(\R_+)^2}  2^{-J}
  \big| \partial_\eta \wt{f}(\eta) \big| \, \big| \varphi_{\sim J}(\s)\wt{f} (\sigma) \big|
  \, K(\xi,\eta,\s) \,d\eta \,d\sigma,
\\
& \lesssim  \sup_{s\approx 2^m} 2^{-J} {\| \partial_\eta \wt{f} \|}_{L^2} 
  \cdot 2^{J/2} {\| \varphi_{\sim J}\wt{f} \|}_{L^2}  
  \\ 
  &  \lesssim 2^{J/2} \sup_{s\approx 2^m} {\| \partial_\eta \wt{f} \|}_{L^2} 
  {\| \partial_\s \wt{f} \|}_{L^2}  \lesssim \e_1^2 2^{-3m\alpha},   
\end{align*}
having applied Hardy's inequality 
${\| \varphi_{\sim J} \wt{f} \|}_{L^2} \lesssim 2^J {\| \partial_\xi \wt{f} \|}_{L^2}$ for $\wt{f}(0)=0$.

\smallskip
\subsubsection*{Estimate of \eqref{regres}}
To estimate these terms we first integrate by parts as in the previous case, 
and then use Schur's lemma 
to take advantage of the restriction on the size of $|\Phi| \lesssim 2^{2J} \approx \js^{-20\alpha}$.
More precisely, integration by parts in $\eta$ and $\s$ gives terms as in \eqref{macreuse1}-\eqref{macreuse3}
with the additional cutoff $\varphi_{<2J}(\Phi)$ and with $\chi_m^l$ replaced by $\chi_m^h$;
we denote these terms by $H_{m,1}, H_{m,2}, H_{m,3}$, respectively.
There are also terms where the cutoff $\varphi_{<2J}(\Phi)$ is hit by the derivatives,
but these are easier to estimate since $\partial_\eta \varphi_{<2J}(\Phi) = 2^{-2J}\varphi_{\sim 2J}(\Phi)\eta/\jeta$
(and similarly for $\s$)
and one can use the factor $\eta/\jeta$ to repeat the integration by parts; 
so we disregard them.

To estimate the term $H_{m,1}$ we first assume, without loss of generality because of symmetry, 
that $|\eta| \geq |\s|$. We then bound
\begin{align*}
& |H_{m,1}(t,\xi)| \lesssim 
  \int_0^t \frac{1}{s} 
  \int_{\R_+} K_1(\xi,\eta) \, |\partial_\eta \wt{f}(\eta)| \,d\eta 
  \, \tau_m(s) \,ds 
  \\
& K_1(\xi,\eta) := 
  \int_{\R_+}\chi^{h}_m(\xi,\eta,\sigma) \varphi_{<2J}(\Phi(\xi,\eta,\s))
  \, |\partial_\sigma \wt{f}(\sigma)|\,d\sigma.
\end{align*}
having used 
$|\xi||\mathfrak{q}'| \lesssim 1$. 
Then we note that for any fixed $|\eta| \gtrsim 2^J$, the region in $\xi$ with
$|\xi|\gtrsim 2^J$ and $|\Phi(\xi,\eta,\s)| \lesssim A$, with $A\lesssim 2^{2J}$,
has size less than $C A 2^{-J}$, with a symmetric estimate exchanging the roles of $\xi$ and $\eta$.
Therefore, we have
\begin{align*}
\int |K_1(\xi,\eta)| \, d\xi + \int |K_1(\xi,\eta)| \, d\eta 
  \lesssim 2^{2J} \cdot 2^{-J} \cdot {\| \partial_\s \wt{f} \|}_{L^1} \lesssim \e_1 2^{-9\alpha m}.
\end{align*}
By applying Schur's Lemma we obtain the bound 
$${\| H_{m,1}(t,\xi)\|}_{L^2} \lesssim 2^{-9\alpha m} {\| \partial_\eta \wt{f} \|}_{L^2} 
  \lesssim \e_1^2 2^{-8\alpha m}.$$ 

The term $H_{m,2}$ can be estimated in a similar way, 
using $2^{-J} \| \varphi_{\sim J}\wt{f} \|_{L^1} \lesssim \e_1$. 

The term $H_{m,3}$ can also be estimated similarly, since
\begin{align*}
& |H_{m,3}(t,\xi)| \lesssim 
  \int_0^t \frac{1}{s} 
  \int_{\R_+} K_3(\xi,\eta) \, |\partial_\eta \wt{f}(\eta)| \,d\eta \, \tau_m(s) \,ds,
  \\
& K_3(\xi,\eta) := 
   \varphi_{>J} (\xi) \varphi_{>J} (\eta) \int_{\R_+} 
  2^{-J} |\varphi_{\sim J} (\sigma)| \varphi_{<2J}(\Phi(\xi,\eta,\s))
  \, |\wt{f}(\sigma)|\,d\sigma,
\end{align*}
and we have
\begin{align*}
\int |K_3(\xi,\eta)| \, d\xi + \int |K_3(\xi,\eta)| \, d\eta 
  \lesssim 2^{2J} \cdot 2^{-J} \cdot {\| \wt{f} \|}_{L^\infty} \lesssim \e_1 2^{-10\alpha m}.
\end{align*}

\smallskip
\subsubsection*{Estimate of \eqref{regnonres}}
In this case we can integrate by parts in $s$ using $e^{is \Phi} = (i\Phi)^{-1}\partial_s e^{is\Phi}$,
and the lower bound on $|\Phi| \gtrsim 2^{2J}$.
This gives several terms, but it suffices to focus on the most relevant one, namely
\begin{equation}\label{cormoran0}
\int_0^t \iint_{(\R_+)^2} e^{is \Phi(\xi,\eta,\sigma)} \xi s \,
 \frac{\varphi_{\geq 2J}(\Phi)}{\Phi} \, \partial_s \wt{f}(\eta) \wt{f} (\sigma)
 \, \mathfrak{q}(\xi,\eta,\sigma) \,d\eta \,d\sigma \, \tau_m(s) \, ds,
\end{equation}
disregarding the symmetric one where $\partial_s$ hits the other profile,
and the easier ones where it hits $s, \tau_m$ and that can be estimated using integration by parts in both $\eta$ and $\s$.

Integrating by parts in $\sigma$ in \eqref{cormoran0} gives the terms
\begin{subequations}
\begin{align}  
\label{cormoran1}
R_{m,1} & := \int_0^t \iint_{(\R_+)^2} e^{is \Phi} \frac{\xi\eta}{\jeta} \,
 \, \partial_s \wt{f}(\eta) \wt{f} (\sigma)
 \, \partial_\s \Big[ \frac{\varphi_{\geq 2J}(\Phi)}{\Phi} \chi^{h}_m(\xi,\eta,\sigma) \mathfrak{q}'(\xi,\eta,\sigma) \Big]
 \,d\eta \,d\sigma \, \tau_m(s) \, ds,
\\
\label{cormoran2} 
R_{m,2} & := \int_0^t \iint_{(\R_+)^2} e^{is \Phi} \frac{\xi\eta}{\jeta} \,
 \,  \frac{\varphi_{\geq 2J}(\Phi)}{\Phi} \chi^{h}_m(\xi,\eta,\sigma) \mathfrak{q}'(\xi,\eta,\sigma)
 \, \partial_s \wt{f}(\eta) \partial_\s \wt{f} (\sigma)
 \,d\eta \,d\sigma \, \tau_m(s) \, ds.
\end{align}
\end{subequations}
To treat \eqref{cormoran2} we first write it as
\begin{align*}
& R_{m,2}(t,\xi) = \int_0^t \iint_{(\R_+)^2} e^{is \Phi} \mathfrak{m}(\xi,\eta,\s) \, \partial_s \wt{f}(\eta) 
 \, \partial_\s \wt{f} (\sigma) \,d\eta \,d\sigma \, \tau_m(s) \, ds,
 \\
& \mathfrak{m}(\xi,\eta,\s) := \frac{\xi\eta}{\jeta} \,
 \,  \frac{\varphi_{\geq 2J}(\Phi)}{\Phi} \chi^{h}_m(\xi,\eta,\sigma) \mathfrak{q}'(\xi,\eta,\sigma).
\end{align*}
%
%
Note that we can estimate, for $a+b+c \leq 6$,
\begin{align}\label{regR2m}
| \partial_\xi^a \partial_\eta^b \partial_\s^c \mathfrak{m}(\xi,\eta,\s) | 
  \lesssim 2^{-2J} \cdot 2^{-2J(a+b+c)} \lesssim 2^{140\alpha m}.
\end{align}
This allows us to use the following classical lemma

\begin{lem}\label{lemBm}
Let $\mathfrak{m}$ be supported on $(\xi,\eta,\sigma) \in [t_1, t_1 + r_1] 
\times [t_2, t_2 + r_2] \times [t_3, t_3 + r_3]$ and satisfy, for  $a,b,c \in \{ 0,1,2 \}$,
$$
| \partial_\xi^a \partial_\eta^b \partial_\sigma^c \mathfrak{m}(\xi,\eta,\sigma) | \leq r_1^{-a} r_2^{-b} r_3^{-c}.
$$
Denote
$$
B_\mathfrak{m} (f, g) (x) := \widehat{\mathcal{F}}_{\xi \rightarrow x}^{-1} 
\iint \mathfrak{m} (\xi, \eta, \sigma) \widehat{f}(\eta)  \widehat{g}(\sigma)  d\eta \, d\sigma. 
$$ 
Then the operator norm of $B_\mathfrak{m}$ can be bounded by
$$
\| B_\mathfrak{m} \|_{L^p \times L^q \to L^r} \lesssim 1
$$
if $(p,q,r) \in [1,\infty]$ satisfy $\frac{1}{p} + \frac{1}{q} = \frac{1}{r}$.
\end{lem}

Letting $\delta$ be small, and denoting $\infty- := \delta^{-1}$ and $2+ := 2(1-2\delta)^{-1}$, 
we estimate using Lemma \ref{lemBm} and \eqref{regR2m}, the boundedness of wave operators and Sobolev's embedding:
\begin{align*}
{\left\| R_{m,2} \right\|}_{L^2} 
  & \lesssim  \int_{s \approx 2^m} 2^{140 \alpha m} \cdot 
  {\big\| \whatF^{-1} \wtF  e^{-is B} \partial_s f \big\|}_{L^{\infty-}} 
  \cdot {\| \whatF^{-1} \partial_\sigma \widetilde f \|}_{L^{2+}} \,ds
  \\
  & \lesssim 2^m  2^{140 \alpha m} \sup_{s \approx 2^m} \cdot 
  {\big\| e^{-is B} \partial_s f \big\|}_{L^{\infty-}} 
  \cdot {\| \jsig \partial_\sigma \widetilde f \|}_{L^2} 
  \\ 
  & \lesssim 2^m \cdot 2^{140 \alpha m} \cdot \e_1^2 2^{-(3/2)m + 10\alpha m} \cdot \e_1 2^{\alpha m}
\end{align*}
where, for the last inequality we have used interpolation between \eqref{dsfL^infty} and \eqref{dsfL^2},
picking $\delta$ small enough.
The bound above for $R_{m,2}$ is more than sufficient compared to the right-hand side of \eqref{secregmain}
if $\alpha$ is small enough. 

The remaining term \eqref{cormoran1} can be handled similarly, since the 
bilinear symbol is just $\partial_\s \mathfrak{m}$ which satisfies estimates like \eqref{regR2m}
with an additional harmless factor of $2^{-2J} = 2^{20\alpha m}$.

\medskip
\section{The main singular interaction}\label{secCS}
\label{TMSI}
We consider here the singular cubic terms defined in \eqref{CubicS}. 
Their kernel contains either a Dirac $\delta(p)$, or a principal value of $\frac{1}{p}$. 
We focus on the latter case, which is slightly more involved, and thus consider that
\begin{align}\label{Cubic2ot1}
\begin{split}
\mathcal{C}^S(a,b,c)(t,\xi)
&:= \iiint e^{it \Phi_{\iota_1\iota_2 \iota_3}(\xi,\eta,\sigma,\theta)} 
  \mathfrak{m}(\xi,\eta,\sigma,\theta) \wt{a}(t,\eta) \wt{b}(t,\sigma)  
  \wt{c}(t,\theta) \frac{\widehat{\phi}(p)}{p} \, d\eta \, d\sigma \,d\theta,
\\
\Phi_{\iota_1\iota_2 \iota_3}(\xi,\eta,\sigma,\theta) & = \jxi - \iota_1 \jeta
  - \iota_2 \langle \sigma \rangle - \iota_3 \langle \theta \rangle, 
\qquad p = \xi - \lambda \eta - \mu \sigma - \nu \theta.
\end{split}
\end{align}
We aim to show
\begin{align}\label{Cubic2otest}
{\Big\| \jxi \partial_\xi \int_0^t \mathcal{C}^S(f, f, f)(s) \, ds \Big\|}_{L^2}
  \lesssim \e_1^3 \jt^{\alpha}.
\end{align}

Observe that
\begin{align}\label{Cubotherid}
(\jxi \partial_\xi + X_{\eta,\s,\theta}) \Phi_{\iota_1, \iota_2 ,\iota_3}(\xi,\eta,\sigma,\theta) = p,
 \qquad X_{\eta,\s,\theta} := \lambda \iota_1\jeta \partial_\eta + \mu \iota_2\jsig \partial_\sigma
  + \nu \iota_3\langle \theta \rangle \partial_\theta.
\end{align}
When applying $\jxi \partial_\xi$ to \eqref{Cubic2ot1}, we can use this identity 
to integrate by parts in $\eta,\sigma$ and $\theta$.
Since the adjoint satisfies $X_{\eta,\s,\theta}^\ast = - X_{\eta,\s,\theta} + \{ \mbox{lower order terms} \}$ 
we see that
\begin{subequations}\label{Cubic2ot5}
\begin{align}\label{Cubic2ot5.1}
& \jxi \partial_\xi\mathcal{C}^S(f,f,f)(t,\xi)
= i t \iiint e^{it \Phi_{\iota_1\iota_2 \iota_3}(\xi,\eta,\sigma,\theta)} \mathfrak{m}(\xi,\eta,\sigma,\theta) 
\, \wt{f}(\eta) \wt{f}(\sigma)  \wt{f}(\theta) 
  \widehat{\phi}(p) \, d\eta \, d\sigma \,d\theta
\\
\label{Cubic2ot5.2}
& + \iiint e^{it \Phi_{\iota_1\iota_2 \iota_3}(\xi,\eta,\sigma,\theta)} \mathfrak{m}(\xi,\eta,\sigma,\theta) 
  X_{\eta,\s,\theta} \big(\wt{f}(\eta) \wt{f}(\sigma)  \wt{f}(\theta) \big)
  \frac{\widehat{\phi}(p)}{p} \, d\eta \, d\sigma \,d\theta
\\
\label{Cubic2ot5.3}
& + \iiint e^{it \Phi_{\iota_1\iota_2 \iota_3}(\xi,\eta,\sigma,\theta)} \mathfrak{m}(\xi,\eta,\sigma,\theta) 
\, \wt{f}(\eta) \wt{f}(\sigma)  \wt{f}(\theta) 
  \, \big( \jxi \partial_\xi + X_{\eta,\s,\theta} \big) \Big[ \frac{\widehat{\phi}(p)}{p}\Big] \, d\eta \, d\sigma \,d\theta 
  \\
  \label{Cubic2ot5.4}
  & +  \iiint e^{it \Phi_{\iota_1\iota_2 \iota_3}(\xi,\eta,\sigma,\theta)} 
  (\langle \xi \rangle \partial_\xi + X_{\eta,\s,\theta}) \, \mathfrak{m}(\xi,\eta,\sigma,\theta) 
\, \wt{f}(\eta) \wt{f}(\sigma)  \wt{f}(\theta) 
  \, \frac{\widehat{\phi}(p)}{p} \, d\eta \, d\sigma \,d\theta 
  \\
 & \nonumber
 + \{ \mbox{lower order terms} \}.
\end{align}
\end{subequations}

\medskip
\noindent
{\it Estimate of \eqref{Cubic2ot5.1}}.
The first term in \eqref{Cubic2ot5} does not have a singular kernel and can be estimated
integrating by parts in the ``uncorrelated'' variables $\eta,\sigma$ and $\theta$,
relying on the vanishing of $\widetilde{f}(0)$ 
and of $\mathfrak{m}(\xi,\eta,\sigma,\theta)$ when $\eta$, $\sigma$, or $\theta$ is zero, see \eqref{cubicSm'}.
(Each integration in one of the three variables is also similar to the argument in the proof of \eqref{locdec1}.)
These integration by parts arguments give the bound $\| \eqref{Cubic2ot5.1} \|_{L^2} \lesssim \e_1^3 \jt^{-2+3\alpha}$,
consistent with \eqref{Cubic2otest}.


\medskip
\noindent
{\it Estimate of \eqref{Cubic2ot5.2}}.
For this term is suffices to use the H\"older estimate \eqref{lemCSend} in Lemma \ref{lemCS},
estimating in $L^2$ the profile that is hit by the derivative, and the other two in $L^\infty$:
$$
\left\| \int_0^t \eqref{Cubic2ot5.2} \,ds \right\|_{L^2} 
  \lesssim \int_0^t \| \jxi \partial_\xi \wt{f} \|_{L^2} 
  \frac{\e_1^2}{\js} \,ds \lesssim \e_1^3 \int_0^t \js^{\alpha-1} \,ds \lesssim \e_1^3 \jt^\alpha.
$$

\medskip
\noindent
{\it Estimate of \eqref{Cubic2ot5.3}}.
For this term we observe, see \eqref{Cubotherid} and \eqref{Cubic2ot1}, that
\begin{align*}
\big( \jxi \partial_\xi + X_{\eta,\s,\theta} \big) p =  \Phi_{\iota_1\iota_2\iota_3}(\xi,\eta,\s,\theta),
\end{align*}
hence
\begin{align}\label{Cubic2ot10}
\big( \jxi \partial_\xi + X_{\eta,\s,\theta} \big) \Big[ \frac{\widehat{\phi}(p)}{p}\Big] 
  = \Phi_{\iota_1\iota_2\iota_3}(\xi,\eta,\s,\theta) \partial_p \Big[ \frac{\widehat{\phi}(p)}{p}\Big]
\end{align}
in the sense of distributions.
We can then use the factor of $\Phi_{\iota_1\iota_2\iota_3}$ in the right-hand side above
to integrate by parts in $s$, and obtain
\begin{align}\label{Cubic2ot11}
\int_0^t i \eqref{Cubic2ot5.3} \, ds & = 
\iiint e^{is \Phi_{\iota_1\iota_2\iota_3}} \mathfrak{m}(\xi,\eta,\sigma,\theta)
\, \wt{f}(\eta) \wt{f}(\sigma)  \wt{f}(\theta) 
  \, \partial_p \frac{\widehat{\phi}(p)}{p}\, d\eta \, d\sigma \,d\theta \, \Big|_{s=0}^{s=t}
\\
\label{Cubic2ot12}
&   - \int_0^t \iiint e^{is \Phi_{\iota_1\iota_2\iota_3}} \mathfrak{m}(\xi,\eta,\sigma,\theta)
\, \partial_s \Big[ \wt{f}(\eta) \wt{f}(\sigma)  \wt{f}(\theta) \Big]
  \partial_p \frac{\widehat{\phi}(p)}{p} \, d\eta \, d\sigma \,d\theta \, ds. 
\end{align}
To estimate \eqref{Cubic2ot11} we convert the $\partial_p$ into $\partial_\eta$ and
integrate by parts in $\eta$. This gives
\begin{align*}
& \iiint e^{is \Phi_{\iota_1\iota_2\iota_3}} \mathfrak{m}(\xi,\eta,\sigma,\theta)
\, \wt{f}(\eta) \wt{f}(\sigma)  \wt{f}(\theta) 
  \, \partial_p \frac{\widehat{\phi}(p)}{p}\, d\eta \, d\sigma \,d\theta \, \Big|_{s=0}^{s=t}  \\
& \qquad\qquad \qquad=  - \iiint e^{is \Phi_{\iota_1\iota_2\iota_3}} \partial_\eta \mathfrak{m}(\xi,\eta,\sigma,\theta)
\, \wt{f}(\eta) \wt{f}(\sigma)  \wt{f}(\theta) 
  \,  \frac{\widehat{\phi}(p)}{p}\, d\eta \, d\sigma \,d\theta \, \Big|_{s=0}^{s=t}   \\   
&  \qquad\qquad \qquad\quad- \iiint e^{is \Phi_{\iota_1\iota_2\iota_3}} \mathfrak{m}(\xi,\eta,\sigma,\theta)
\,   \partial_\eta \wt{f}(\eta) \wt{f}(\sigma)  \wt{f}(\theta) 
  \,  \frac{\widehat{\phi}(p)}{p}\, d\eta \, d\sigma \,d\theta \, \Big|_{s=0}^{s=t}   \\   
  & \qquad\qquad \qquad\quad - \iiint i s  \partial_\eta \Phi  e^{is \Phi_{\iota_1\iota_2\iota_3}} \mathfrak{m}(\xi,\eta,\sigma,\theta)
\,    \wt{f}(\eta) \wt{f}(\sigma)  \wt{f}(\theta) 
  \,  \frac{\widehat{\phi}(p)}{p}\, d\eta \, d\sigma \,d\theta \, \Big|_{s=0}^{s=t}  
\end{align*}
Out of the three terms on the right-hand side, the worst is the last one when $s=t$, 
since it picked up a factor $t$ when the derivative hit the complex exponential. 
The $L^2$ norm of this term can be estimated using Lemma \ref{lemCS} (see also the comment below its statement)
by
$$
C t \frac{\varepsilon_1^2}{\jt} {\| f \|}_{L^2} \lesssim \varepsilon_1^3 \jt^\alpha,
$$
as desired.
We also refer the reader to \cite[Section 9.3]{KGV1d} for arguments similar to those above.

The term \eqref{Cubic2ot12} is similar to the one just treated. 
We may assume that $\partial_s$ hits $\wt{f}(\sigma)$.
Again we convert $\partial_p$ into $\partial_\eta$ and integrate by parts in $\eta$.
This causes a loss of $s$ when hitting the exponential phase
which is offset by an $L^\infty\times L^2\times L^\infty$ estimate
with $\partial_s \wt{f}$ placed in $L^2$ and giving $\e_1^2 \js^{-1+\alpha/2}$ decay by \eqref{dsfL^2}.

\medskip
\noindent
{\it Estimate of \eqref{Cubic2ot5.4}}. This can be treated through Lemma \ref{lemCS}.
The only small difficulty is the loss of derivatives resulting from the differentiation of the symbol
(see the description after \eqref{cubicSm'}),
but this is easily recovered using the $H^4$ a priori bound from Proposition \ref{propbootstrap}, and $p_0 < \alpha$.


\medskip
\section{The remainder terms}\label{secred}
We now explain how to deal with the remainder terms appearing after the renormalization \eqref{Renodtf}
and the various splittings of the nonlinear terms.


\smallskip
\subsection{Weighted norm and bootstrap for $g$}\label{bootg}
The a priori bootstrap estimates for $g=f+T(g,g)$ in Proposition \ref{propbootstrap} 
can be closed fairly easily using the slightly better a priori bounds on $f$.
To explain this, let us first observe that for all practical purposes
one can think that the operator $T$, defined in \eqref{defTiota}-\eqref{profile1}, 
essentially has the form 
\begin{align}\label{bootgTsimp0}
T(a,b) = e^{itB} \big( e^{-itB}a \, e^{-itB} b \big),     
\end{align}
and, in particular, satisfies standard H\"older bounds of the form
\begin{align}\label{bootgTsimp}
{\| e^{-itB} T(a,b) \|}_{L^r} \lesssim {\| e^{-itB}a \|}_{L^q} {\| e^{-itB}b \|}_{L^p}, \qquad 1/r=1/p+1/q,
  \quad r,p,q\in(1,\infty).
\end{align}
The validity of \eqref{bootgTsimp} follows from the fact that
the symbol in \eqref{defTiota} is a $\delta$ plus a 
$p.v.$ contribution near its singularity, see \eqref{propmuS}, 
divided by a smooth non-vanishing expression. We refer the reader to Section 6 of \cite{KGV1d}
for detailed lower bounds on $\Phi_{\iota_1\iota_2}$ and related symbol-type estimates;
in particular, see Subsection 6.3 in \cite{KGV1d} for the proof of the bound \eqref{bootgTsimp}
(in fact, a stronger version of it with a gain of almost a full derivative).
With the above notation simplification, we may assume that \eqref{profile1} reads
\begin{align}\label{gboot}
g = f + T(g,g) = f + e^{itB} (e^{-itB}g)^2.
\end{align}

The Sobolev type estimate for $g$ is easy to obtain directly from \eqref{gboot} and the  
a priori Sobolev bound for $f$ and $g$ in \eqref{bootstrap0}.
First, observe that by fixing an arbitrary small number $\delta>0$,
using Sobolev's embedding, and the boundedness of wave operators (see \eqref{Wopdef}) 
on $L^p$, we can estimate
\begin{align*}
 {\big\| e^{-itB} g \big\|}_{L^\infty} 
  \lesssim {\big\| e^{-itB} g \big\|}_{W^{\delta+,1/\delta}}
  \lesssim {\big\| e^{-it\langle \partial_x \rangle} \mathcal{W}^\ast g \big\|}_{W^{\delta+,1/\delta}}
  \lesssim \e_1 \jt^{-\frac{1}{2}+p_0};
\end{align*}
for the last inequality we have used interpolation between the bounds in the second line of \eqref{bootstrap0}.
Then we can estimate $\| B^4 (e^{-itB}g)^2 \|_{L^2} \lesssim \| B^4 g \|_{L^2} \| e^{-itB}g \|_{L^\infty}
\lesssim \e_1^2 \jt^{2p_0-1/2}$.

For the remaining $L^\infty$ estimate in \eqref{bootstrap0}
we fix again an arbitrary small number $\delta$
and similarly to above use Sobolev's embedding, followed by the boundedness of wave operators and 
the natural version of \eqref{bootgTsimp} with derivatives, to obtain
\begin{align*}
{\big\| e^{-it\langle \partial_x \rangle} 
  \mathbf{1}_{\pm}(D) \mathcal{W}^* T(g,g) \big\|}_{L^\infty} 
  & \lesssim
  {\big\| e^{-it\langle \partial_x \rangle} 
  \mathbf{1}_{\pm}(D) \mathcal{W}^* T(g,g) \big\|}_{W^{\delta+,1/\delta}}
  \\
  & \lesssim {\big\| e^{-itB} T(g,g) \big\|}_{W^{\delta+,1/\delta}}
  \lesssim \| e^{-itB} B^{\delta+}g \|^2_{L^{2/\delta}} \lesssim \e_1^2 \jt^{-1+2p_0},
\end{align*}
where we again used Sobolev-Gagliardo-Nirenberg interpolation for the last bound.
This proves the bounds for $g$ in \eqref{eqbootstrap}.


Let us also mention for later use that from \eqref{gboot} one can 
establish a weak weighted bound for $g$ of the form
\begin{align}\label{dxig}
{\| \jxi \partial_\xi \wt{g} \|}_{L^2_\xi} & \lesssim \e_1 \jt^{1/2+\alpha/10};
\end{align}
this is because applying $\partial_\xi$ to $\wt{T} (g, g)$ costs a factor of $t$ and $\|( e^{-itB}g )^2\|_{L^2}
\lesssim \e_1^2\jt^{-1/2}$.
The bound \eqref{dxig} is also helpful in estimating weighted norms of remainder terms, see Subsection \ref{wrem} below.
For more details on the bootstrap for $g$ and \eqref{dxig}, 
we refer the reader to Subsections 7.1 and 7.2 in \cite{KGV1d}.

\smallskip
\subsection{Weighted estimates for remainder terms}\label{wrem}
The remainder terms $\mathcal{R}$ that arise when we substitute $f$ for $g$ into the leading quadratic and cubic terms,
see \eqref{R}, are not hard to estimate in our current functional framework,
while their estimates are more lengthy in \cite{KGV1d}. We provide a few details below.
The two terms on the right-hand side of \eqref{RQ} 
can be thought of as cubic versions of the smooth quadratic terms $\mathcal{Q}_R$,
and essentially are of the form
\begin{align}\label{FRQ-0}
\mathcal{R}_Q[a,b,c](t,\xi) :=  \int_0^t \int_{(\R_+)^3} 
  e^{is \Phi_{\iota_1 \iota_2 \iota_3}(\xi,\eta,\sigma,\theta)} \wt{a}_{\iota_1}(\eta) 
  \wt{b}_{\iota_2}(\sigma) \wt{c}_{\iota_3}(\theta)
  \mathfrak{q}(\xi,\eta,\sigma,\theta) \,d\eta \, d\sigma \, d\theta  ds,
  \\
  \nonumber
  \mbox{for} \quad (a,b,c) = (g,g,g) \quad \mbox{or}  \quad (f,g,g),
\end{align} 
where $\Phi_{\iota_1\iota_2\iota_3}$ is as in \eqref{CubicS}, 
and $\mathfrak{q}$ is some smooth symbol that vanishes whenever $\eta\cdot\sigma\cdot\theta=0$.
To estimate $\mathcal{R}_Q[g,g,g]$ we can substitute one $g$ with $f$ via \eqref{gboot}, 
up to terms that are easier to estimate; for example, they are better versions of 
the terms \eqref{CS4} which we will discuss below.
We can then reduce matters to considering $\mathcal{R}_Q[f,g,g]$, 
and similarly substituting one more $g$ with $f$ 
we further reduce to $\mathcal{R}_Q[f,f,g]$.
For the cubic expression $\mathcal{R}_Q[f,f,g]$ we can apply simple integration by parts arguments in
each of the variables in distorted Fourier space, like those performed in Section \ref{sectionregular}
(see for example \eqref{reglow} and the expression resulting after integration by parts \eqref{macreuse1}-\eqref{macreuse3}).
With this operation, and using \eqref{dxig},
we obtain the bound $$\| \jxi \partial_\xi \mathcal{R}_Q[f,f,g](t) \|_{L^2} \lesssim s \, (\e_1 s^{-1+\alpha})^2 \cdot \e_1 s^{-1/2+\alpha/10},$$
and it follows that
%
%
\begin{align}\label{estRQ}
{\Big\| \jxi \partial_\xi \int_0^t \mathcal{R}_Q[f,g,g](s,\xi) \, ds \Big\|}_{L^2} \lesssim \e_1^3. 
\end{align}


Moving on to $\mathcal{R}_C$ in \eqref{RC}, we first observe that, in view of \eqref{CubicR}-\eqref{CubicRm}, 
the higher order remainder terms corresponding to $\mathcal{C}^R$ are quartic version of \eqref{FRQ-0},
and therefore easier to handle compare to the quartic terms we discuss below. 
The contributions corresponding to $\mathcal{C}^S$ 
are, up to a permutation of the inputs, of the form $\mathcal{C}^S(T(g,g),a,b)$ with $a,b = f$ or $g$.
In view of \eqref{gboot} and the H\"older estimates of Lemma \ref{lemCS},
we see that all these are essentially standard four-fold products of the form
\begin{align}\label{CS4}
\mathcal{R}_C[g,g,a,b] 
  \sim e^{itB}(e^{\iota_1itB}g \cdot e^{\iota_2itB}g 
  \cdot e^{\iota_3itB}a \cdot e^{\iota_4itB}b\big), \quad
  \mbox{for} \quad a,b = f \, \mbox{or} \, g,
\end{align}
where $\iota_j$, $j=1,\dots,4$ are signs.
Note that \eqref{CS4} is a quartic version of $\mathcal{C}^S$. 
Also note that the hardest term if the one with all $g$ inputs, $\mathcal{R}_C[g,g,g,g]$. 
For this, we can substitute one $g$ with $f$ 
via \eqref{gboot}
up to terms that are quintic and easier to estimate, reducing matters to 
estimating $\mathcal{R}_C[g,g,g,f]$;
then, similarly substituting one more $g$ with $f$ we can reduce to 
$\mathcal{R}_C[g,g,f,f]$.
To estimate $\mathcal{R}_C[g,g,f,f]$ 
we can use a ``commutation identity'' similar to the one in \eqref{Cubotherid}
which we used for the main singular cubic term.
In particular, we can distribute the $\partial_\xi$ derivative on the Fourier side 
and obtain identities analogous to those in \eqref{Cubic2ot5} 
(with four input function instead of three).
We can then estimate very similarly to Section \ref{TMSI} 
using \eqref{dxig}, the a priori bounds \eqref{bootstrap0}, 
\eqref{dsfL^2} and an identical estimate for $\partial_s g$, 
arriving at the bound 
\begin{align}\label{CS5}
{\Big\| \jxi \partial_\xi \int_0^t \mathcal{R}_C[g,g,f,f] 
  \, ds \Big\|}_{L^2}
  \lesssim \int_0^t {\big\| \partial_\xi \wt{g} \big\|}_{L^2} \cdot (\e_1 \js^{-1/2})^3 \, ds + \e_1^4
  \lesssim \e_1^4 \jt^{\alpha/10}.
\end{align}


\smallskip
\subsection{The Sobolev norm}\label{Sobolev}
The bootstrap on the Sobolev-type norm $\| \jxi^4 \wt{f} \|_{L^2}$ is substantially 
easier to obtain than the weighted bound.
The estimate for the cubic terms follows directly from Lemma \ref{lemCS} and the natural extension of \eqref{lemCS1} 
with derivatives; see \cite[Lemma 6.13]{KGV1d}.

A slightly less immediate argument is required to estimate the quadratic terms $\mathcal{Q}^R$.
For this we need to consider a term of the form
\begin{align}\label{Sobolev1} 
\begin{split}
I(s,\xi)
  & := \iint e^{is ( \jxi -   \jeta -  \jsig) } \, \langle \xi \rangle^4 \,\mathfrak{q}(\xi,\eta,\sigma) 
  \, \wt{f}(\eta) \wt{f}(\sigma) \, d\eta \, d\sigma
  \\
  & = 
  \iint_{\R^2_+} e^{is ( \jxi -   \jeta -   \jsig) } \wt{f}(\eta) \wt{f}(\sigma) \, \frac{\eta}{\jeta} \frac{\s}{\jsig} \frac{\langle \xi \rangle^4}{\jeta+\jsig}
  \mathfrak{q}'(\xi,\eta,\sigma) \,d\eta \, d\sigma,
\end{split}
\end{align}
see \eqref{QR}-\eqref{qsymbol}, where $\mathfrak{q}'$ satisfies \eqref{q'symbol}.
It suffices to bound \eqref{Sobolev1} in $L^2$ by $\e_1^2 \js^{-1+p_0}$.
Note that since $p_0 \ll \alpha$ this does not follow immediately from integration by parts in frequency.
We then proceed as follows.
First, we may assume, by the symmetry in $\eta$ and $\s$, and the decay of the kernel away from $\xi-\eta-\s=0$, 
that $1 \leq |\xi| \lesssim |\s| \leq |\eta|$.
Second, we can integrate by parts in $\s$ and estimate the main contribution by
\begin{align}\label{Sobolev2}
& \Big| \iint_{\R^2_+} \frac{1}{s} \, e^{is ( \jxi -   \jeta -   \jsig) } \, \wt{f}(\eta) \, \partial_\s \wt{f}(\sigma) \,
  \frac{\eta}{\jeta} \frac{\langle \xi \rangle^4}{\jeta+\jsig} \mathfrak{q}'(\xi,\eta,\sigma) \,d\eta \, d\sigma \Big|.
\end{align}
Then, we see that if we restrict the integration region to $|\eta| \gtrsim \js^{2\alpha}$, 
an application of Young's inequality suffices:
\begin{align*} 
& \quad {\Big\| \iint_{\R^2_+ \cap \{ |\eta| \gtrsim \js^{2\alpha} \}  } \frac{1}{s} \, e^{is ( \jxi -   \jeta -   \jsig) } 
  \, \wt{f}(\eta) \, \partial_\s \wt{f}(\sigma) \,
  \frac{\eta}{\jeta} \frac{\langle \xi \rangle^4}{\jeta+\jsig} \mathfrak{q}'(\xi,\eta,\sigma) \,d\eta \, d\sigma \Big\|}_{L^2} 
  \\
 &  \lesssim \frac{1}{s} 
  \Big\| \iint_{ \R^2_+ \cap \{ |\eta| \gtrsim \js^{2\alpha} \}  } \big| \jeta^4 \wt{f}(\eta) \big| \, \big| \partial_\s \wt{f}(\sigma) \big| 
  \, \frac{1}{\jeta}
  \frac{1}{\langle \xi -\eta -\sigma \rangle^{N}} \,d\eta \, d\sigma  \Big\|_{L^2} 
\\  
 & \lesssim s^{-1-2\alpha} {\| \jxi^4 \wt{f} \|}_{L^2} {\| \partial_\xi \wt{f} \|}_{L^1} \lesssim s^{-1-\alpha+p_0} \e_1^2.
\end{align*}
When instead $|\eta| \lesssim \js^{2\alpha}$, 
we can integrate by parts also in $\eta$ and bound the main contribution in $L^2_\xi$ by
\begin{align*}
& \quad  \Big\| \iint_{ \R^2_+ \cap \{ |\eta| \lesssim \js^{2\alpha} \}  } \frac{1}{s^2} \, e^{is ( \jxi -   \jeta -   \jsig) } \, \partial_\eta \wt{f}(\eta) 
  \, \partial_\s \wt{f}(\sigma) \,
  \frac{\langle \xi \rangle^4}{\jeta+\jsig} \mathfrak{q}'(\xi,\eta,\sigma) \,d\eta \, d\sigma  \Big\|_{L^2} 
  \\ 
 & \lesssim s^{-2}  \Big\|  \iint_{  \R^2_+ \cap \{ |\eta| \lesssim \js^{2\alpha} }  \big| \partial_\eta \wt{f}(\eta) \big| \, 
  \big| \partial_\s \wt{f}(\sigma) \big| 
  \, {\jeta}^3 \frac{1}{\langle \xi -\eta -\sigma \rangle^{N}} \,d\eta \, d\sigma  \Big\|_{L^2} 
  \\
&  \lesssim s^{-2+6\alpha} 
  {\| \partial_\xi \wt{f} \|}_{L^2} {\| \partial_\xi \wt{f} \|}_{L^1}
  \lesssim s^{-2+8\alpha} \e_1^2,
\end{align*}
which is more than sufficient.

The remainder terms \eqref{R} can be handled similarly, and in fact are even easier to estimate 
via direct applications of the H\"older bounds of Lemmas \ref{lemQR} and \ref{lemCS}
and arguing as in Subsection \ref{bootg} to estimate the $T$ operator.

\smallskip
\subsection{Fourier $L^\infty$ estimates and asymptotics}\label{Linfty}
Finally, we briefly discuss how to close the bootstrap for the norm ${\| \jxi^{3/2} \wt{f} \|}_{L^\infty}$ 
in \eqref{eqbootstrap}.
The simple observation is that since our a priori assumptions \eqref{bootstrap0} are stronger than the
assumptions made in \cite{KGV1d}, but the conclusion we want is exactly the same for this norm, 
all of the arguments used in \cite{KGV1d} apply here.
In particular, we can replace the a priori bound (10.31) in \cite[Sec. 10.3]{KGV1d}, that is
(recall the notation \eqref{LPnot})
\begin{equation*}
{\| \varphi_{[-5, 5]} \partial_\xi \widetilde{f} \|}_{L^2} \lesssim \varepsilon_1 \langle t \rangle^{\rho}, 
\end{equation*}
in which $\rho = \alpha + \beta \gamma = 1/4 -$, by the stronger
\begin{equation*}
{\| \varphi_{[-5, 5]} \partial_\xi \widetilde{f} \|}_{L^2} \lesssim \varepsilon_1 \langle t \rangle^{\alpha}, 
\end{equation*}
with $\alpha > p_0$ being small as in Proposition \ref{propbootstrap}.
Then, under the a priori assumption \eqref{bootstrap0} for $\wt{f}$, we can prove
\begin{align}\label{Linfty1}
{\| \jxi^{3/2} \wt{f} \|}_{L^\infty} \leq C\e_0 + C\e_1^2.
\end{align}
The argument used to obtain \eqref{Linfty1} is based on the derivation of an (Hamiltonian) ODE 
for $\wt{f}$ which relies on precise asymptotics for the cubic terms $\mathcal{C}^S$;  
this ODE also gives the asymptotic modified scattering behavior in \eqref{asy}. 
We refer the reader to Proposition 10.1 and \cite[Sec 10]{KGV1d} for detailed statements and proofs.

As it turns out, we believe that the arguments that have been used in \cite{KGV1d} to bound the cubic terms, and obtain
their asympotics in time, cannot be simplified much.
One can, however, simplify the estimates for the other higher order remainder terms $\mathcal{R}(f,g)$,
see \eqref{Renodtf}-\eqref{RC}, as we explain below.

First, we observe that the weighted bound \eqref{QRmainbound} for $\mathcal{Q}^R$ and the 
weighted bound \eqref{estRQ} for $\mathcal{R}_Q$, 
imply the Fourier $L^\infty$ bound for these terms via the interpolation 
$\| \jxi^{3/2} h \|_{L^\infty} \lesssim \| \jxi \partial_\xi h \|_{L^2}^{1/2} \| \jxi^2 h \|_{L^2}^{1/2}$,
and the fact that a bound on $\| \jxi^2 h \|_{L^2}^{1/2}$ which decays in time
is easy to obtain for $h = \mathcal{Q}^R$ or $\mathcal{R}_Q$.
Indeed, one can obtain $\| \jxi^2  \int_0^t \mathcal{Q}^R ds \|_{L^2} \lesssim \e_1^2$ 
proceeding in a similar way as in the estimate of 
$\| \jxi  \mathcal{Q}^R \|_{L^2}$ in Lemma \ref{ibis}; see the proof of \eqref{dsfpr1'}.
Hence, we can bound
\begin{equation*}
\begin{split}
{\Big\| \jxi^{3/2} \int_0^t \mathcal{Q}^R ds \Big\|}_{L^\infty} 
&  \lesssim {\Big\| \jxi \partial_\xi  \int_0^t \mathcal{Q}^R ds \Big\|}_{L^2}^{1/2} 
  {\Big\| \jxi^2  \int_0^t \mathcal{Q}^R ds \Big\|}_{L^2}^{1/2}
  \lesssim \e_1^2.
\end{split}
\end{equation*}
For $\mathcal{R}_Q$, we can obtain $\| \jxi^2  \int_0^t \mathcal{R}_Q ds \|_{L^2} \lesssim \e_1^2 $ 
via a similar argument as in the estimate \eqref{estRQ},
since $\mathcal{R}_Q$ is a cubic version of the quadratic terms $\mathcal{Q}^R$, 
and then conclude using interpolation as above.

Finally, for the term $\mathcal{R}_C$ in \eqref{RC} we can first substitute all the $g$'s by $f$'s, 
noticing that the error terms generated in this process are higher order remainders. 
Then, we can estimate very similarly to Section \ref{TMSI} using the a priori bounds \eqref{bootstrap0}, 
and \eqref{dsfL^2} to obtain the following analogue of \eqref{CS5} with $f$ instead of $g$:
\begin{align*}
{\Big\| \jxi \partial_\xi \int_0^t \mathcal{R}_C[f,f,f,f] 
  \, ds \Big\|}_{L^2}
  \lesssim \int_0^t {\big\| \partial_\xi \wt{f} \big\|}_{L^2} \cdot (\e_1 \js^{-1/2})^3 \, ds + \e_1^4
  \lesssim \e_1^4.
\end{align*}
Then we obtain $\| \jxi^{3/2}  \int_0^t \mathcal{R}_C ds \|_{L^\infty} \lesssim \e_1^4$
via interpolation with a simpler Sobolev bound as above.

All the terms on the right-hand side of \eqref{Renodtf} are thus accounted for,
Proposition \ref{propbootstrap} follows, and the main Theorem \ref{mainthm} is proven.

%
%

\smallskip
\subsection*{Data availability statement}
This manuscript has no associated data.

\end{document}